\documentclass[reqno, a4paper, 11pt]{amsart}
\usepackage{amsfonts, amsmath, amssymb, amsthm}
\usepackage[hmargin=2.75cm, vmargin=2.75cm]{geometry}
\usepackage[dvipsnames]{xcolor}
\usepackage{graphicx, setspace}

\newtheorem{thm}{Theorem}[section]
\newtheorem{lemma}[thm]{Lemma}
\newtheorem{prop}[thm]{Proposition}
\newtheorem{cor}[thm]{Corollary}

\newtheorem{thmx}{Theorem}

\newtheorem{propx}{Proposition}

\theoremstyle{remark}
\newtheorem{rmk}[thm]{Remark}

\theoremstyle{definition}

\newcommand{\bg}{\bar{g}}
\newcommand{\bx}{\bar{x}}
\newcommand{\hh}{\hat{h}}

\newcommand{\ox}{\overline{X}}
\newcommand{\ola}{\overline{\Lambda}}
\newcommand{\oth}{\overline{\Theta}}
\newcommand{\wtf}{\widetilde{F}}
\newcommand{\wtg}{\widetilde{G}}
\newcommand{\wtu}{\widetilde{U}}
\newcommand{\wtv}{\widetilde{V}}
\newcommand{\wtla}{\widetilde{\Lambda}}
\newcommand{\wtth}{\widetilde{\Theta}}

\newcommand{\ga}{\gamma}
\newcommand{\ep}{\epsilon}
\newcommand{\vep}{\varepsilon}
\newcommand{\ls}{\lambda, \sigma}

\newcommand{\pa}{\partial}

\newcommand{\mh}{\mathbb{H}}
\newcommand{\mn}{\mathbb{N}}
\newcommand{\mr}{\mathbb{R}}

\newcommand{\mcc}{\mathcal{C}}
\newcommand{\mcd}{\mathcal{D}}
\newcommand{\mck}{\mathcal{K}}
\newcommand{\mcm}{\mathcal{M}}

\newcommand{\tni}{\textnormal{II}}

\newcommand{\hx}{H^{1,2}(X; \rho^{1-2\ga})}
\newcommand{\dhmr}{\dot{H}^{1,2}(\mr^N_+; x_N^{1-2\ga})}

\renewcommand{\(}{\left(}
\renewcommand{\)}{\right)}

\begin{document}
\title[Conformal metrics with prescribed fractional scalar curvature]{Conformal metrics with prescribed fractional\\
scalar curvature on conformal infinities\\
with positive fractional Yamabe constants}

\author{Seunghyeok Kim}
\address[Seunghyeok Kim]{Department of Mathematics and Research Institute for Natural Sciences, College of Natural Sciences, Hanyang University, 222 Wangsimni-ro Seongdong-gu, Seoul 04763, Republic of Korea}
\email{shkim0401@gmail.com}

\begin{abstract}
Let $(X, g^+)$ be an asymptotically hyperbolic manifold with conformal infinity $(M, [\hh])$.
Our primary aim is to introduce the prescribed fractional scalar curvature problem on $M$
and to provide its solutions under various geometric conditions on $X$ and $M$.
We also deduce the existence results for the fractional Yamabe problem in the end-point cases, e.g., $n = 3$, $\ga = \frac{1}{2}$ and $M$ is non-umbilic, etc.
Finally, we prove that all solutions we find here are smooth on $M$.
\end{abstract}

\date{\today}
\thanks{S. Kim was partially supported by Basic Science Research Program through the National Research Foundation
of Korea(NRF) funded by the Ministry of Education (NRF2017R1C1B5076384).}
\maketitle

\allowdisplaybreaks
\numberwithin{equation}{section}

\section{Introduction}
The main objective of this paper is to introduce and to examine the prescribed fractional scalar curvature problem.
It is the nonlocal version of the prescribed scalar curvature problem
which has been served as one of the central problems in conformal geometry.

\medskip
Suppose that $X^{n+1}$ is an $(n+1)$-dimensional smooth manifold with boundary $M^n$ and $\rho$ is a {\it defining function} for $M$,
namely, a function in $X$ satisfying
\[\rho > 0 \quad \text{in } X, \quad \rho = 0 \quad \text{on } M \quad \text{and} \quad d\rho \ne 0 \quad \text{on } M.\]
We say that a metric $g^+$ in $X$ is {\it conformally compact} if there is a defining function $\rho$ such that $\bg := \rho^2 g^+$ is a smooth compact metric on the closure $\ox$ of $X$.
If $[\hh]$ denotes the conformal class of the metric $\hh = \bg|_M$ on $M$,
then $(M, [\hh])$ is called the {\it conformal infinity} of the manifold $X$.
An {\it asymptotically hyperbolic} manifold is a conformally compact manifold such that all sectional curvatures tend to -1 as a point approaches to $M$.

Assume that $(X, g^+)$ is conformally compact and Einstein. It is called {\it Poincar\'e-Einstein} and known to be a special example of asymptotically hyperbolic manifolds.
In \cite{GZ}, Graham and Zworski introduced the {\it fractional conformal Laplacian} $P^{\ga}[g^+, \hh]$ on the conformal infinity $(M, [\hh])$ for $n > 2\ga$.
It is a pseudo-differential operator which satisfies the conformal covariance property
\begin{equation}\label{eq-cc}
P^{\ga}[g^+, w^{4 \over n-2\ga} \hh] u = w^{-{n+2\ga \over n-2\ga}}P^{\ga}[g^+, \hh] (wu)
\end{equation}
for all $u, w \in C^{\infty}(M)$ such that $w > 0$ on $M$, and
\begin{equation}\label{eq-sym}
\sigma(P^{\ga}[g^+, \hh]) = \sigma((-\Delta_{\hh}))^{\ga})
\end{equation}
where $\sigma$ denotes the principal symbol and $(-\Delta_{\hh})^{\ga}$ is the fractional power of the Laplace-Beltrami operator on $M$.
If $\ga = 1$ or $2$, it agrees with the classical conformal Laplacian and the Paneitz operator, respectively.
More generally, it is the same as the GJMS operator \cite{GJMS} for each $\ga \in \mn$, constructed via the ambient metric.
For general asymptotically hyperbolic manifolds, the fractional conformal Laplacians on their conformal infinities can be still defined
as in Joshi and S\'a Barreto \cite{JS} or Guillarmou \cite{Gui}; see also \cite{MM, CG, GQ}.

Set the {\it fractional scalar curvature} (or {\it $\ga$-scalar curvature}) by $Q^{\ga}[g^+, \hh] := P^{\ga}[g^+, \hh](1)$.
If $(X, g^+)$ is Poincar\'e-Einstein and $\ga = 1$ or $2$,
then it is nothing but the scalar curvature or Branson's $Q$-curvature up to a constant multiple, respectively.
In this paper, we study the {\it prescribed fractional scalar curvature problem} (or the {\it prescribed $\ga$-scalar curvature problem}) for $\ga \in (0,1)$, which addresses:

\medskip
Given a smooth function $f$ on the boundary $M$, can one find a metric $\hh_0 \in [\hh]$ on $M$ whose $\ga$-scalar curvature $Q^{\ga}[g^+, \hh_0]$ is $f$?

\medskip \noindent
By virtue of \eqref{eq-cc}, solving this problem is equivalent to looking for a solution to the equation
\begin{equation}\label{eq-ps}
P^{\ga}[g^+, \hh] u = f u^{n+2\ga \over n-2\ga} \quad \text{and} \quad u > 0 \quad \text{on } (M^n, \hh)
\end{equation}
provided that $n > 2\ga$.

\medskip
This type of the problems dates back to at least the work of Kazdan and Warner in 1970s.
In \cite{KW}, they proved that for a compact manifold $M^n$ with $n \ge 3$, a smooth function $f$ that is somewhere negative on $M$ can be a scalar curvature,
and every smooth function $f$ can be a scalar curvature if and only if $M$ admits a metric whose scalar curvature is positive.
To get these results, they attempted to make use of \eqref{eq-ps} with $\ga = 1$.
Unfortunately, there exists a function $f$ such that the equation does not have a solution.
To resolve this obstacle, they introduced an auxiliary diffeomorphism $\varphi$ on $M$
and replaced the right-hand side of \eqref{eq-ps} with $(f \circ \varphi) u^{n+2\ga \over n-2\ga}$.
This idea gives additional flexibility for the existence of $u$.

Later, by studying the associated constrained minimization problem to \eqref{eq-ps} with $\ga = 1$,
Escobar and Schoen \cite{ES} proved that \eqref{eq-ps} with $\ga = 1$ has a solution
if $(M, \hh)$ is a locally conformally flat compact manifold whose scalar curvature $R[\hh]$ is positive and fundamental group is non-trivial,
$f$ is a function positive somewhere on $M$, and it achieves a global maximum point at which $\nabla^k f = 0$ for $k = 1, \cdots, n-2$.
They also considered when a locally conformally flat manifold $M$ has the vanishing scalar curvature.
Using a similar idea, Aubin and Hebey \cite{AH} and Hebey and Vaugon \cite{HV} studied when the Weyl tensor is not entirely zero on $M$.

For a smooth compact Riemannian manifold $(\ox^{n+1}, \bg)$ with boundary $(M^n, \hh)$, Escobar \cite{Es} introduced the prescribed mean curvature problem,
which is deeply related to \eqref{eq-ps} with $\ga = \frac{1}{2}$, and solved it under various geometric settings.
For instance, if $X$ has the positive Sobolev quotient, then one can find a solution provided that either
\begin{itemize}
\item[-] $f$ is positive somewhere on $M$ and achieves a global maximum at a non-umbilic point $y$
    where
    \[\frac{-\Delta_{\hh} f(y)}{f(y)} < c \|\tni - H\hh \|_{\hh}^2(y)\]
    for some sufficiently small constant $c > 0$; or
\item[-] $X$ is locally conformally flat but not conformally diffeomorphic to the Euclidean ball, $M$ is umbilic,
    $f$ is a function positive somewhere on $M$, and it achieves a global maximum point at which $\nabla^k f = 0$ for $k = 1, \cdots, n-1$.
\end{itemize}
In the above, $\tni$ is the second fundamental form on $(M, \hh) \subset (\ox, \bg)$, $H$ is the mean curvature and $\|\cdot\|_{\hh}$ is the tensor norm.

As we will see, our theorems provide extensions of the above mentioned results,
which cover all $\ga \in (0,1)$ and most of situations such that the fractional Yamabe constant $\Lambda^{\ga}(M, [\hh])$, which will be defined in \eqref{eq-Y-const}, is positive.

\medskip
If $(X, g^+)$ is the Poincar\'e ball and its conformal infinity $(M^n, \hh)$ is the standard sphere, Eq. \eqref{eq-ps} is said to be the {\it fractional Nirenberg problem}.
By means of the stereographic projection, it is reduced to the problem
\begin{equation}\label{eq-ps-ni}
(-\Delta)^{\ga} u = f u^{n+2\ga \over n-2\ga} \quad \text{and} \quad u > 0 \quad \text{on } \mr^n.
\end{equation}
For $\ga \in (0,1)$, by employing the Caffarelli-Silvestre extension \cite{CS}, Jin et al. \cite{JLX} showed that the solution set of \eqref{eq-ps-ni} is compact
if $f$ is positive somewhere on $\mr^n$ and has flatness order greater than $n-2\ga$ at each critical point\footnote{A function $f \in C(M)$ is said to have {\it flatness order grater than $d > 0$} at a point $y \in M$,
if for some local coordinate system $\bx = (x_1, \cdots, x_n)$ of $M$ centered at $y$,
there exists a neighborhood $\mathcal{N}$ of $0$ such that $f(\bx) = f(0) + o(|\bx|^d)$ in $\mathcal{N}$.}.
Also, in \cite{JLX2}, they deduced an existence result by applying the compactness theorem of \cite{JLX} and the degree counting argument;
compare this with our Theorem \ref{thm-ps-PE}, especially, paying attention to the flatness order of $f$.
On the other hand, the authors in \cite{AC, CLZ} obtained some existence criterions involving a topological condition on the level set of $f$,
by establishing Euler-Hopf type index formulae.
Furthermore, Abdelhedi et al. \cite{ACH} dealt with the case when the flatness order of $f$ is in $(1, n-2\ga]$.
For $\ga \in (0, \frac{n}{2})$, Jin et al. \cite{JLX3} recently extended the results of \cite{JLX, JLX2} by analyzing \eqref{eq-ps-ni}
through its integral representations instead of appealing the extension theorem.

\medskip
If $f$ is a constant on $M$, we call \eqref{eq-ps} the {\it fractional Yamabe problem} (or the {\it $\ga$-Yamabe problem}).
As in the classical case $\ga = 1$, if $M$ is the standard sphere, the solution set of \eqref{eq-ps} or \eqref{eq-ps-ni} with $f = 1$
consists of the bubbles $w_{\ls}$ defined in \eqref{eq-bubble-2}; refer to \cite[Theorem 1.8]{JLX} for its proof.
Their scaling invariance induces the loss of compactness of the Sobolev embedding $H^{\ga}(\Omega) \hookrightarrow L^{2n \over n-2\ga}(\Omega)$ on a smooth bounded domain $\Omega$.
For general manifolds and $\ga \in (0,1)$, \eqref{eq-ps} has been investigated by several researchers \cite{GQ, GW, KMW, MN, DSV} and most of cases are covered up to now.
Also, Qing and Raske \cite{QR} studied it assuming that $\ga \in [1, \frac{n}{2})$ and $M$ is locally conformally flat manifold
with positive Yamabe constant and Poincar\'e exponent less than $\frac{n-2\ga}{2}$.

The next theorem describes the result of Kim et al. \cite{KMW}, which generalizes the pioneering works of Gonz\'alez and Qing \cite{GQ} and Gonz\'alez and Wang \cite{GW}.
\begin{thmx}[Kim, Musso, Wei \cite{KMW}]\label{thm-Y}
Assume that the first $L^2(X)$-eigenvalue $\lambda_1(-\Delta_{g+})$ of the Laplace-Beltrami operator $-\Delta_{g+}$ satisfies
\begin{equation}\label{eq-eig}
\lambda_1(-\Delta_{g+}) > {n^2 \over 4} - \ga^2.
\end{equation}
Suppose also that $\rho$ is the {\it geodesic defining function} in $(X, g^+)$ associated to $(M, \hh)$,
which is a unique defining function splitting the metric $\bg = \rho^2 g^+$ into $d\rho^2 + h_{\rho}$ near $M$
with a family of metrics $\{h_{\rho}\}_{\rho}$ on $M$ such that $h_0 = \hh$.
If one of the following conditions
\begin{enumerate}
\item[(a)] $n \ge 2$, $\ga \in (0, \frac{1}{2})$ and the boundary $(M, \hh)$ of $(\ox, \bg)$ has a point at which the mean curvature $H$ is negative;
\item[(b)] $n \ge 4$, $\ga \in (0, 1)$, $(M, \hh)$ is the non-umbilic boundary of $(\ox, \bg)$ and
    \begin{equation}\label{eq-R-dec}
    R[g^+] + n(n+1) = o(\rho^2) \quad \text{as } \rho \to 0 \text{ uniformly on } M
    \end{equation}
    where $R[g^+]$ denotes the scalar curvature of $(X, g^+)$;
\item[(c)] $n > 3 + 2\ga$, $\ga \in (0,1)$, $(M, \hh)$ is the umbilic boundary of $(\ox, \bg)$
    such that a covariant derivative $R_{\rho\rho;\rho}[\bg]$ of the component $R_{\rho\rho}[\bg]$ of the Ricci tensor on $(\ox ,\bg)$ is negative somewhere on $M$, and \eqref{eq-R-dec} holds;
\item[(d)] $n > 4 + 2\ga$, $\ga \in (0,1)$, $(M, \hh)$ is the umbilic non-locally conformally flat boundary of $(\ox, \bg)$ and
    \begin{equation}\label{eq-R-dec-2}
    R[g^+] + n(n+1) = o(\rho^4) \quad \text{as } \rho \to 0 \text{ uniformly on } M
    \end{equation}
\end{enumerate}
is satisfied, then the $\ga$-Yamabe problem is solvable.
\end{thmx}
\begin{rmk}
For the precise meaning of the solvability of the $\ga$-Yamabe problem, refer to Theorem \ref{thm-reg} below. We have two additional remarks on the previous theorem.

\medskip \noindent
(1) The sign of $H$ or $R_{\rho\rho;\rho}[\bg]$ at a given point on $M$ is {\it intrinsic}, namely, independent of the choice of a representative of the conformal class $[\hh]$ on $M$.
The proof of this fact is given in \cite[Lemma 2.3]{GQ} and \cite[Lemma 2.3]{GW}, respectively.
On the other hand, if \eqref{eq-R-dec} is valid, then $H = 0$ on $M$. Refer to our Lemma \ref{lemma-g-mtr}.

\medskip \noindent
(2) Our condition \eqref{eq-R-dec-2} on (d) is weaker than the corresponding one in \cite[(1.19)]{KMW}.
In fact, one can prove \cite[Lemma 3.2]{KMW} by employing \eqref{eq-R-dec-2} only,
while the argument presented in \cite{KMW} requires \cite[(1.19)]{KMW}.
See Lemma \ref{lemma-g-mtr-2} for details.
\end{rmk}

The first main theorem of this paper deals with the existence of positive solutions to \eqref{eq-ps}
provided that one of the geometric assumptions (a), (b), (c) and (d) in Theorem \ref{thm-Y} is valid and the function $f$ has a suitable behavior.
\begin{thm}\label{thm-ps}
Assume that \eqref{eq-eig} holds, $\Lambda^{\ga}(M, [\hh]) > 0$ and $f$ is a smooth function positive somewhere on $M$. If one of the following conditions
\begin{enumerate}
\item[(A)] condition \textnormal{(a)} holds;
\item[(B)] condition \textnormal{(b)} holds, $f$ achieves a global maximum point at a nonumbilic point $y$ of $M$ and
    \begin{equation}\label{eq-cond-1}
    {-\Delta_{\hh} f(y) \over f(y)} < c^1_{n, \ga} \|\tni\|_{\hh}^2(y)
    \end{equation}
    for some constant $c^1_{n, \ga} > 0$ depending only on $n$ and $\ga$;
\item[(C)] condition \textnormal{(c)} holds, $f$ achieves a global maximum point at $y \in M$, $-\Delta_{\hh} f(y) = 0$ and $R_{\rho\rho;\rho}[\bg](y) < 0$;
\item[(D)] condition \textnormal{(d)} holds, $f$ achieves a global maximum point at a non-locally conformally flat point $y$ of $M$,
    \begin{equation}\label{eq-cond-2}
    -\Delta_{\hh} f(y) = 0 \quad \text{and} \quad {-(-\Delta_{\hh})^2 f(y) \over f(y)} < c^2_{n, \ga} \|W\|_{\hh}^2(y)
    \end{equation}
    for some constant $c^2_{n, \ga} > 0$ depending only on $n$ and $\ga$
\end{enumerate}
is satisfied, then the prescribed $\ga$-scalar curvature problem \eqref{eq-ps} is solvable. Here
\begin{enumerate}
\item[-] $\Lambda^{\ga}(M, [\hh]) > 0$ is the fractional Yamabe constant whose definition is introduced in \eqref{eq-Y-const};
\item[-] $-\Delta_{\hh}$ is the Laplace-Beltrami operator on $(M, \hh)$, a positive semi-definite operator;
\item[-] $\tni$ is the second fundamental form on $(M, \hh) \subset (\ox, \bg)$ and $\|\tni\|_{\hh}$ is its 2-tensor norm;
\item[-] $W$ is the $(0, 4)$ Weyl tensor on $(M, \hh)$ and $\|W\|_{\hh}$ is its 4-tensor norm.
\end{enumerate}
\end{thm}
\begin{rmk}
For the precise meaning of the solvability of the $\ga$-scalar curvature problem, refer to Theorem \ref{thm-reg} below.
We have two additional remarks on the previous theorem.

\medskip \noindent (1) Clearly, it holds that $-\Delta_{\hh} f(y) \ge 0$ for any local maximum point $y \in M$ of $f$.

\medskip \noindent (2) In \eqref{eq-c_1} and \eqref{eq-c_2}, we present explicit values of the positive constants $c^1_{n,\ga}$ and $c^2_{n,\ga}$.
It is interesting to know if the suggested values are somehow optimal.
\end{rmk}

In fact, we can extend the above theorems to the end-point case.
First, inspired by the works of Marques \cite{Ma} and Almaraz \cite{Al} for the boundary Yamabe problem,
we can prove the following result on the $\frac{1}{2}$-Yamabe problem.
It validates the expectation in Remarks 1.2 (4) and 1.4 (3) of \cite{KMW}.
\begin{thmx}\label{thm-Y-end}
Suppose that $\ga = \frac{1}{2}$ and \eqref{eq-eig} is valid. If one of the following conditions
\begin{enumerate}
\item[(b$^\prime$)] $n = 3$, $(M, \hh)$ is the non-umbilic boundary of $(\ox, \bg)$ and \eqref{eq-R-dec} holds;
\item[(c$^\prime$)] $n = 4$, $(M, \hh)$ is the umbilic boundary of $(\ox, \bg)$, $R_{\rho\rho;\rho}[\bg]$ is negative at some point on $M$ and \eqref{eq-R-dec} holds;
\item[(d$^\prime$)] $n = 5$, $(M, \hh)$ is the umbilic non-locally conformally flat boundary of $(\ox, \bg)$ and \eqref{eq-R-dec-2} holds
\end{enumerate}
is satisfied, then the $\ga$-Yamabe problem is solvable.
\end{thmx}
\noindent Unlike the boundary Yamabe problem, the $\frac{1}{2}$-Yamabe problem allows us
to choose a metric only on the conformal infinity $M$.
Theorem \ref{thm-Y-end} confirms that the decay assumptions on the scalar curvature $R[g^+]$ given as \eqref{eq-R-dec} and \eqref{eq-R-dec-2} take away the difference of these problems.

Based on the previous theorem, we are able to deduce the following result.
\begin{thm}\label{thm-ps-end}
Assume that $\ga = \frac{1}{2}$, \eqref{eq-eig} holds, $\Lambda^{\ga}(M, [\hh]) > 0$ and $f$ is a smooth function positive somewhere on $M$.
If one of the following conditions
\begin{enumerate}
\item[(B$^{\prime}$)] condition \textnormal{(b$^{\prime}$)} holds and $f$ achieves a global maximum point at a nonumbilic point $y$ of $M$;
\item[(C$^{\prime}$)] condition \textnormal{(c$^{\prime}$)} holds, $f$ achieves a global maximum point at $y \in M$, $-\Delta_{\hh} f(y) = 0$ and $R_{\rho\rho;\rho}[\bg](y) < 0$;
\item[(D$^{\prime}$)] condition \textnormal{(d$^{\prime}$)} holds, $f$ achieves a global maximum point at a non-locally conformally flat point $y$ of $M$ and $-\Delta_{\hh} f(y) = 0$
\end{enumerate}
is satisfied, then the prescribed $\ga$-scalar curvature problem \eqref{eq-ps} is solvable.
\end{thm}
\begin{rmk}
Our proof for the above theorem also produces analogous results to (B), (B$^{\prime}$), (D) and (D$^{\prime}$) for the prescribed mean curvature problem, which extend the work of Escobar \cite{Es}.
In \cite[Theorem 3.3]{Es}, the result corresponding to (B) was obtained for the prescribed mean curvature problem under the additional assumption that $n \ge 6$.
It is notable that $\max_M f = 1$ is implicitly assumed in condition (3.20) of \cite{Es}.

For the classical prescribed scalar curvature problem, which corresponds to $\ga = 1$, an analogous result to (D) was obtained by Aubin and Hebey \cite{AH}.
Our method gives the way to compute the explicit value of $c_{n,1}^2$, which was not treated in their paper.
See Hebey and Vaugon \cite{HV} for further investigation in this direction.
\end{rmk}

Suppose now that $(X, g^+)$ is Poincar\'e-Einstein and $(M, \hh)$ is locally conformally flat\footnote{Recall that all 2-dimensional Riemannian manifolds are locally conformally flat.}.
Let $\rho$ be the geodesic defining function in $X$ associated to $(M, \hh)$.
In \cite[Proposition 1.5]{KMW}, it was shown that for each $y \in M$, there exists the {\it Green's function} $G(\cdot,y)$ on $\ox \setminus \{y\}$ which solves
\[\begin{cases}
-\text{div}_{\bg}(\rho^{1-2\ga}\nabla G(\cdot,y)) + E(\rho)\, G(\cdot,y) = 0 &\text{in } (X, \bg),\\
\pa^{\ga}_{\nu} G(\cdot,y) := - \kappa_{\ga} \(\lim\limits_{\rho \to 0+} \rho^{1-2\ga} \dfrac{\pa G(\cdot,y)}{\pa \rho}\) = \delta_y &\text{on } M
\end{cases}\]
in the distribution sense where $\kappa_{\ga} := \frac{2^{2\ga}\Gamma(\ga)}{2\Gamma(1-\ga)} > 0$ and $\delta_y$ is the Dirac measure at $y$.
Consult Proposition \ref{prop-GQ-ext} below for the motivation of the definition of our Green's function.
Moreover, it was shown in \cite[Corollary 6.1]{MN} that the Green's function can be written as
\begin{equation}\label{eq-Green}
G(x,y) = g_{n,\ga}\, d_{\bg}(x,y)^{-(n-2\ga)} + A + \Psi(d_{\bg}(x,y))
\end{equation}
for any $x \in \ox$ near $y \in M$, where $g_{n,\ga} > 0$ is a suitable normalizing constant, $A \in \mr$ and $\Psi$ is a function in $\mr$ satisfying
\[|\Psi(t)| \le C|t|^{\min\{1,2\ga\}} \quad \text{and} \quad |\Psi'(t)| \le C|t|^{\min\{0,2\ga-1\}}\]
for $t$ small.
Under a technical condition $A > 0$, we can deduce the following theorem.
\begin{thm}\label{thm-ps-PE}
Suppose that $\ga \in (0,1)$, \eqref{eq-eig} holds, $(X, g^+)$ is a Poincar\'e-Einstein manifold, $\Lambda^{\ga}(M, [\hh]) > 0$ and $f$ is a smooth function positive somewhere on $M$.
Assume also that the constant $A$ in \eqref{eq-Green} is positive.
If one of the following conditions
\begin{enumerate}
\item[(E)] $n > 2\ga$, $(M,\hh)$ is locally conformally flat,
    $f$ achieves a global maximum $y \in M$ and it has flatness order greater than $n-2\ga$ at $y$; or
\item[(E$^{\prime}$)] $n = 2$, $f$ achieves a global maximum $y \in M$ and it has flatness order greater than $2-2\ga$ at $y$
\end{enumerate}
is satisfied, then the prescribed $\ga$-scalar curvature problem \eqref{eq-ps} is solvable.
\end{thm}
\begin{rmk}
(1) Suppose that $\ga = 1$ and either $n \le 7$ or $M$ is locally conformally flat.
In this situation, the positive mass theorem of Schoen and Yau \cite{SY1, SY2} implies that $A \ge 0$,
and the condition $A > 0$ holds if and only if $M$ is not conformally equivalent to the standard sphere $\mathbb{S}^n$.
Currently, formulating and proving an analogue of the positive mass theorem for the case $\ga \in (0,1)$ is left as a challenging open problem.

\medskip \noindent (2) Condition \textnormal{(E$^{\prime}$)} is automatically satisfied if $n = 2$ and $\ga \in [\frac{1}{2}, 1)$.

\medskip \noindent (3) The flatness condition on $f$ in (E) and (E$^{\prime}$) corresponds to the results of Jin, Li and Xiong \cite[Theorems 1.2, 1.3]{JLX}
for the fractional Nirenberg problem on the standard $n$-dimensional unit sphere.
\end{rmk}

The proof of Theorems \ref{thm-ps}, \ref{thm-ps-end}, \ref{thm-ps-PE} and \ref{thm-Y-end} is based on the constrained minimization technique
and the Chang-Gonz\'alez extension theorem stated in Proposition \ref{prop-GQ-ext}.
In particular, the statement in the above theorems that `the $\gamma$-Yamabe problem or the prescribed $\ga$-scalar curvature problem is solvable'
actually means the existence of a weak solution to the equation
\begin{equation}\label{eq-ps-ext}
\begin{cases}
-\textnormal{div}_{\bg} (\rho^{1-2\ga} \nabla U) + E(\rho)U = 0 \quad \text{in } (X, \bg),\\
U > 0 \quad \text{in } \ox, \quad U = u \quad \text{on } M,\\
\pa^{\ga}_{\nu} U = - \kappa_{\ga} \(\lim\limits_{\rho \to 0+} \rho^{1-2\ga} \dfrac{\pa U}{\pa \rho}\) = f u^{n+2\ga \over n-2\ga} \quad \text{on } M
\end{cases}
\end{equation}
that belongs the weighted Sobolev space $\hx$, a space defined in the next paragraph to Proposition \ref{prop-GQ-ext}.
The following regularity theorem assures that it is a classical solution to \eqref{eq-ps} under the general condition $H = 0$ on $M$.
\begin{thm}\label{thm-reg}
Suppose that $\ga \in (0,1)$, \eqref{eq-eig} holds, $\Lambda^{\ga}(M, [\hh]) > 0$, $f \in C^{\infty}(M)$ and the mean curvature $H$ on $(M, \hh) \subset (\ox, \bg)$ vanishes.
Then the trace $u$ on $M$ of each weak solution $U \in \hx$ to \eqref{eq-ps-ext} is contained in the space $C^{\infty}(M)$
and so solves \eqref{eq-ps} in a classical sense.
\end{thm}
\noindent This result has been regarded to be true since the study on the fractional Yamabe problem was initiated.
See, for example, \cite[Proposition 3.2]{GQ}.
However, we need to pay more attention to it since the component $\bg_{(n+1)(n+1)}$ of our metric $\bg$ on $\ox$ may not be constant;
compare with the settings of \cite{CaS, JLX, FF, CSt} where regularity results of nonlocal problems were obtained.
Indeed, due to this fact, the equation of a difference quotient of $U$ with respect to a tangential direction to $M$ (see \eqref{eq-reg-21})
may have an inhomogeneous term involving the normal derivative $\pa_\rho U$ of $U$, which is not easy to handle directly.
To bypass this issue, we shall pick a representative of the conformal class $[\hh]$ on $M$
such that the exponential map at some point is a local volume preserving map, whose existence is guaranteed by Cao \cite{Ca} and G\"unther \cite{Gun}.
Then it will suffice to control $\rho \pa_{\rho} U$ instead of $\pa_{\rho} U$, which is a rather simple task owing to the scaling property of \eqref{eq-ps-ext}.

\medskip
The paper is organized as follows. In Section \ref{sec-pre}, we introduce background materials such as the Chang-Gonz\'alez extension theorem,
the weighted Sobolev space $\hx$ and the fractional Yamabe constant $\Lambda^{\ga}(M, [\hh])$.
In Sections \ref{sec-ps-cri} and \ref{sec-ps-ex}, we derive a sufficient condition to ensure the existence of a solution to \eqref{eq-ps}
and describe several situations when the condition holds, proving Theorems \ref{thm-ps}, \ref{thm-ps-end} and \ref{thm-ps-PE}.
Section \ref{sec-reg} is devoted to regularity of solutions to \eqref{eq-ps}, and especially, the proof of Theorem \ref{thm-reg}.
Finally, in Appendix \ref{sec-Y}, we prove Theorem \ref{thm-Y-end} by employing various arguments from the papers \cite{GQ, KMW, Ma, Al} on the fractional and boundary Yamabe problem.

\medskip \noindent \textbf{Notations.}

\medskip \noindent - We use $2^* := \frac{n+2\ga}{n-2\ga}$ and Einstein summation convention throughout the paper.

\medskip \noindent - We always assume that $1 \le i, j, k, l \le n$ and $1 \le a, b \le N := n+1$.

\medskip \noindent - $C > 0$ is a generic constant which can vary from line to line.

\medskip \noindent - For any $x \in \overline{\mr^N_+}$ and $r > 0$, we set by $B^N(x,r)$ the $N$-dimensional open ball of radius $r$ and center $x$, and $B^N_+(x,r) = B^N(x,r) \cap \overline{\mr^N_+}$.

\medskip \noindent - Let $C^{\alpha}(\Omega) = C^{\lfloor \alpha \rfloor, \alpha - \lfloor \alpha \rfloor}(\Omega)$ for a number $0 < \alpha \notin \mn$ and a subset $\Omega$ of $\mr^n$.
Here $\lfloor \alpha \rfloor$ denotes the largest integer that does not exceed $\alpha$.

\section{Preliminaries}\label{sec-pre}
\subsection{Extension results, functional spaces and the fractional Yamabe constant}
Let $n > 2\ga$, $\ga \in (0,1)$, $(X^{n+1}, g^+)$ be an asymptotically hyperbolic manifold, $(M^n, [\hh])$ its conformal infinity
and $\rho$ a geodesic defining function of $(M, \hh)$.
Then $\bg = \rho^2 g^+$ becomes a smooth metric on $\ox$.
We also write $P^{\ga}[g^+, \hh]$ to denote the fractional conformal Laplacian on $(M, \hh)$.
Since the metric $g^+$ in $X$ is always fixed, we will use a simplified notation $P^{\ga}_{\hh} = P^{\ga}[g^+, \hh]$.

\medskip
We first recall the local interpretation of the operator $P^{\ga}_{\hh}$. To this end, we need to introduce a space
\begin{equation}\label{eq-mcd}
\mcd := \left\{ U = u + v \rho^{2\ga} + o(\rho^{2\ga}) \in C^{\infty}(X) \cap C^0(\ox) : u, v \in C^{\infty}(M) \right\}.
\end{equation}

\begin{propx}[Chang, Gonz\'alez \cite{CG}]\label{prop-GQ-ext}
Let $H$ be the mean curvature on $(M, \hh) \subset (\ox, \bg)$. Define
\[E(\rho) = \rho^{-1-s} (-\Delta_{g^+} - s(n-s)) \rho^{n-s} \quad \text{in } X \quad \text{where } s := {n \over 2} + \ga.\]
Then it can be shown to be \begin{equation}\label{eq-E}
E(\rho) = - \({n-2\ga \over 2}\) \({\pa_{\rho} \sqrt{|\bg|} \over \sqrt{|\bg|}}\) \rho^{-2\ga} \quad \text{in } M \times (0,r_0)
\end{equation}
for some small $r_0 > 0$; refer to \cite[Remark 2.2]{CK} for its derivation.
Suppose that \eqref{eq-eig} holds, $H = 0$ on $M$ for $\ga \in (\frac{1}{2}, 1)$ and a function $u \in C^{\infty}(M)$ is given.

\medskip \noindent \textnormal{(1)} Let $U \in \mcd$ be a solution to
\begin{equation}\label{eq-ext-1}
\begin{cases}
-\textnormal{div}_{\bg} (\rho^{1-2\ga} \nabla U) + E(\rho)U = 0 &\text{in } (X, \bg),\\
U = u &\text{on } M
\end{cases}
\end{equation}
whose unique existence is guaranteed in \cite{JS, GZ, CG}. Then it holds that
\begin{equation}\label{eq-ext-2}
\pa^{\ga}_{\nu} U = - \kappa_{\ga} \(\lim\limits_{\rho \to 0+} \rho^{1-2\ga} \dfrac{\pa U}{\pa \rho}\) =  \begin{cases}
P_{\hh}^{\ga} u &\text{for } \ga \in (0,1) \setminus \left\{ \dfrac{1}{2} \right\},\\
P_{\hh}^{\ga} u - \(\dfrac{n-1}{2}\) Hu &\text{for } \ga = \dfrac{1}{2}.
\end{cases}
\end{equation}

\medskip \noindent \textnormal{(2)} There is a defining function $\rho^*$ such that
\begin{equation}\label{eq-adaptive}
E(\rho^*) = 0 \text{ in } X \quad \text{and} \quad \rho^*(\rho) = \rho (1+O(\rho^{2\ga})) \text{ near } M,
\end{equation}
which is called adaptive. In addition, if we denote $\wtu = (\rho/\rho^*)^{n-2\ga \over 2} U \in \mcd$, then it solves
\begin{equation}\label{eq-ext-3}
\begin{cases}
-\textnormal{div}_{\bg^*} \((\rho^*)^{1-2\ga} \nabla \wtu\) = 0 \quad \text{in } (X, \bg^*),\\
\pa^{\ga}_{\nu} \wtu = - \kappa_{\ga} \(\lim\limits_{\rho^* \to 0+} (\rho^*)^{1-2\ga} \dfrac{\pa \wtu}{\pa \rho^*}\) = P_{\hh}^{\ga} u - Q^{\ga}_{\hh} u \quad \text{on } M
\end{cases}
\end{equation}
where $\bg^* := (\rho^*)^2 g^+$ and $Q^{\ga}_{\hh}$ is the fractional scalar curvature.
\end{propx}

Define the {\it weighted Sobolev space} $\hx$ with weight $\rho^{1-2\ga}$ as the completion of the space $\mcd$ in \eqref{eq-mcd} with respect to the norm
\begin{equation}\label{eq-norm}
\|U\|_{\hx} := \(\int_X \rho^{1-2\ga} (|\nabla U|_{\bg}^2 + U^2) dv_{\bg}\)^{\frac{1}{2}}.
\end{equation}
By \eqref{eq-adaptive}, $H^{1,2}(X; (\rho^*)^{1-2\ga})$ is a Hilbert space equivalent to $\hx$.

In addition, for any element $U \in \hx$, we set the energy $I^{\ga}(U)$ associated to \eqref{eq-ext-1} and \eqref{eq-ext-2} by
\begin{equation}\label{eq-I^ga}
\begin{aligned}
&\ I^{\ga}(U) \\
&= \begin{cases}
\kappa_{\ga} \int_X (\rho^{1-2\ga} |\nabla U|_{\bg}^2 + E(\rho) U^2) dv_{\bg} &\text{for } \ga \in (0,1) \setminus \left\{ \dfrac{1}{2} \right\},\\
\kappa_{\ga} \int_X (\rho^{1-2\ga} |\nabla U|_{\bg}^2 + E(\rho) U^2) dv_{\bg} + \(\dfrac{n-1}{2}\) \int_M HU^2 dv_{\hh} &\text{for } \ga = \dfrac{1}{2},
\end{cases}
\end{aligned}
\end{equation}
and the energy $J^{\ga}(U)$ associated to \eqref{eq-ext-3} by
\begin{equation}\label{eq-J^ga}
J^{\ga}(U) = \kappa_{\ga} \int_X (\rho^*)^{1-2\ga} |\nabla U|_{\bg}^2 \, dv_{\bg} + \int_M Q^{\ga}_{\hh} U^2 dv_{\hh}.
\end{equation}

Let $H^{\ga}(M)$ be the fractional Sobolev space realized as the space of the traces of functions in $\hx$.
Then it holds that $H^{\ga}(M) \hookrightarrow L^{2^*+1}(M)$ and the {\it fractional Yamabe constant}
\begin{equation}\label{eq-Y-const}
\Lambda^{\ga}(M, [\hh]) = \inf_{u \in H^{\ga}(M) \setminus \{0\}} {\int_M u P^{\ga}_{\hh} u dv_{\hh} \over (\int_M |u|^{2^*+1} dv_{\hh})^{n-2\ga \over n}}
\end{equation}
is well-defined. If we also set
\begin{equation}\label{eq-Y-const-1}
\ola^{\ga}(X, [\hh]) =
\inf_{U \in \hx \setminus \{0\}} {I^{\ga}(U) \over (\int_M |U|^{2^*+1} dv_{\hh})^{n-2\ga \over n}}
\end{equation}
and
\[\wtla^{\ga}(X, [\hh]) =
\inf_{U \in \hx \setminus \{0\}} {J^{\ga}(U) \over (\int_M |U|^{2^*+1} dv_{\hh})^{n-2\ga \over n}},\]
then the result of Case \cite{Cs} shows that
\begin{equation}\label{eq-Y-const-2}
\Lambda^{\ga}(M, [\hh]) = \ola^{\ga}(X, [\hh]) = \wtla^{\ga}(X, [\hh]) > -\infty
\end{equation}
under the validity of \eqref{eq-eig}. We remark that only Poincar\'e-Einstein manifolds were treated in \cite{Cs},
but the arguments in it can be generalized to arbitrary asymptotically hyperbolic manifolds as far as it holds that $H = 0$ for $\ga \in (\frac{1}{2}, 1)$.

\medskip
Fix any $\ga \in (0,1) \setminus \{\frac{1}{2}\}$ and assume that $(M, [\hh])$ is a conformal infinity with positive fractional Yamabe constant.
In the remaining part of this subsection, we briefly explain how to extend the operator $P_{\hh}^{\ga}: C^{\infty}(M) \to C^{\infty}(M)$
to $P_{\hh}^{\ga}: H^{\ga}(M) \to H^{-\ga}(M) := (H^{\ga}(M))^*$.

\begin{lemma}\label{lemma-norm}
If $\ga \in (0,1) \setminus \{\frac{1}{2}\}$, \eqref{eq-eig} holds, $\Lambda^{\ga}(M, [\hh]) > 0$, and $H = 0$ on $M$ for $\ga \in (\frac{1}{2}, 1)$,
then $I^{\ga}(U) \ge 0$ for all $U \in \hx$.
Furthermore, $\|U\|_* := \sqrt{I^{\ga}(U)}$ serves as an equivalent norm to the standard $\hx$-norm introduced in \eqref{eq-norm}.
\end{lemma}
\begin{proof}
The first claim trivially comes from \eqref{eq-Y-const-1} and \eqref{eq-Y-const-2}.
The proof of the second claim can be found in \cite[Lemma 3.4]{CDK}.
\end{proof}
\begin{lemma}\label{lemma-GQ-ext}
Given $\ga \in (0,1) \setminus \{\frac{1}{2}\}$, the operator $P_{\hh}^{\ga}: H^{\ga}(M) \to H^{-\ga}(M)$ is well-defined and bounded.
Moreover, if \eqref{eq-eig} holds, $\Lambda^{\ga}(M, [\hh]) > 0$, and $H = 0$ on $M$ for $\ga \in (\frac{1}{2}, 1)$,
then for any $u \in H^{\ga}(M)$, there exists a unique function $U \in \hx$ such that \eqref{eq-ext-1} and \eqref{eq-ext-2} are valid in weak sense.
\end{lemma}
\begin{proof}
Given any fixed element $u_0 \in H^{\ga}(M)$, let $\{u_m\}_{m \in \mn} \subset C^{\infty}(M)$ such that $u_m \to u_0$ in $H^{\ga}(M)$.
For each $v \in H^{\ga}(M)$, we define
\[\int_M (P_{\hh}^{\ga} u_0) v\, dv_{\hh} := \lim_{m \to \infty} \int_M (P_{\hh}^{\ga} u_m) v\, dv_{\hh}.\]
Then \eqref{eq-sym} tells us that it is well-defined and in fact independent of a choice of $\{u_m\}_{m \in \mn}$.
Also, $P_{\hh}^{\ga}: H^{\ga}(M) \to H^{-\ga}(M)$ is clearly bounded.

For each $u_m \in C^{\infty}(M)$, let $U_m \in \mcd$ be a unique solution to \eqref{eq-ext-1}. We get from Lemma \ref{lemma-norm} that
\[C\|U_m\|_{\hx}^2 \le I^{\ga}(U_m) = \int_M (P_{\hh}^{\ga} u_m) u_m dv_{\hh} \le C\sup_{m \in \mn} \|u_m\|_{H^{\ga}(M)}^2 < \infty.\]
Therefore, $U_m \rightharpoonup U_0$ in $\hx$ for some $U_0 \in \hx$. It is plain to verify that $U_0$ weakly solves \eqref{eq-ext-1} and \eqref{eq-ext-2}  with $u = u_0$.
The uniqueness of $U_0$ again results from Lemma \ref{lemma-norm}.
\end{proof}
\noindent It is notable that the conformal covariance property \eqref{eq-cc} still holds for all $u \in H^{\ga}(M)$.

\subsection{Bubbles}
Let $N = n+1$, $\mr^N_+ = \{(\bx, x_N): \bx \in \mr^n,\, x_N > 0\}$ and $\dhmr$ be the {\it homogeneous weighted Sobolev space} with weight $x_N^{1-2\ga}$,
i.e., the completion of the space $C^{\infty}_c(\overline{\mr^N_+})$ with respect to norm
\[\|U\|_{\dhmr} := \(\int_{\mr^N_+} x_N^{1-2\ga} |\nabla U(\bx,x_N)|^2 dx\)^{\frac{1}{2}}.\]
As can be seen in \cite[Proposition 1.3]{GQ}, there exists the optimal constant $S_{n,\ga} > 0$ depending only on $n \in \mn$ and $\ga \in (0,1)$ such that
\begin{equation}\label{eq-sob-t}
\|U(\cdot, 0)\|_{L^{2^*+1}(\mr^n)}^2 \le S_{n,\ga} \|U\|_{\dhmr}^2
\end{equation}
for any function $U \in \dhmr$.

Given any $\lambda > 0$ and $\sigma \in \mr^n = \pa \mr^N$, let $w_{\ls}$ be a function on $\mr^n$ defined by
\begin{equation}\label{eq-bubble-2}
w_{\ls}(\bx) := \alpha_{n,\ga} \(\lambda \over \lambda^2 + |\bx - \sigma|^2 \)^{n-2\ga \over 2}
\quad \text{where } \alpha_{n,\ga} := 2^{n-2\ga \over 2} \left[ \frac{\Gamma(\frac{n+2\ga}{2})} {\Gamma(\frac{n-2\ga}{2})} \right]^{n-2\ga \over 4\ga}
\end{equation}
and $W_{\ls}$ the unique solution in $\dhmr$ of the equation
\begin{equation}\label{eq-bubble}
\begin{cases}
-\text{div}(x_N^{1-2\ga} \nabla W_{\ls}) = 0 &\text{in } \mr^N_+,\\
W_{\ls} = w_{\ls} &\text{on } \mr^n.
\end{cases}
\end{equation}
Then
\[\pa_{\nu}^{\ga} W_{\ls} := - \kappa_{\ga} \(\lim\limits_{x_N \to 0+} x_N^{1-2\ga} \dfrac{\pa W_{\ls}}{\pa x_N}\)
= (-\Delta)^{\ga} w_{\ls} = w_{\ls}^{2^*} \quad \text{on } \mr^n\]
(see \cite{CS}) and the equality of \eqref{eq-sob-t} is attained by $U = cW_{\ls}$ for any $c > 0$, $\lambda > 0$ and $\sigma \in \mr^n$.
If $\ga = \frac{1}{2}$, the function $W_{\ls}$ can be explicitly written as
\begin{equation}\label{eq-bubble-12}
W_{\ls}(x) = \alpha_{n, \frac{1}{2}} \({\lambda \over (\lambda + x_N)^2+ |\bx - \sigma|^2}\)^{n-1 \over 2} \quad \text{for } x = (\bx, x_N) \in \mr^N_+.
\end{equation}
Throughout the paper, we call either $w_{\ls}$ or $W_{\ls}$ a {\it bubble}.

It can be easily checked that $W_{\lambda, 0}(\cdot, x_N)$ is radially symmetric for each $x_N > 0$.
Moreover, an immediate consequence of \eqref{eq-Y-const-2} and \eqref{eq-sob-t} is that
\begin{equation}\label{eq-H^n-ene}
\ola^{\ga}(\mh^N, [\hh_c]) = S^{-1}_{n,\ga} \kappa_{\ga} = \(\int_{\mr^n} w_{\ls}^{2^*+1} d\bx\)^{2\ga \over n}
\end{equation}
where $(\mh^N, g^+ = \frac{d\bx^2 + dx_N^2}{x_N^2})$ is the Poincar\'e half-space model and $\hh_c := x_N^2 g^+|_{\mr^n} = d\bx^2$.

\subsection{Results on the metric}
Let us recall the expansion of the metric $\bg$ on $\ox$ near the boundary $M$. Its proof can be found in Escobar \cite[Lemma 3.1]{Es2}.
\begin{lemma}\label{lemma-bg-exp}
Given any point $y \in M$, let $x = (\bx, x_N) \in \mr^N_+$ be the Fermi coordinate on $\ox$ around $y$. Then it holds that
\[\bg^{ij}(x) = \delta_{ij} + 2\tni_{ij} x_N + {1 \over 3} R_{ikjl}[\hh] x_kx_l + \bg^{ij}_{\phantom{ij},Nk} x_Nx_k + (3\tni_{ik}\tni_{kj} + R_{iNjN}[\bg]) x_N^2 + O(|x|^3)\]
and
\[\sqrt{|\bg|}(x) = 1 - nHx_N + {1 \over 2} \(n^2 H^2 - \|\tni\|^2 - R_{NN}[\bg]\) x_N^2 - nH_{,i} x_ix_N - {1 \over 6} R_{ij}[\hh] x_ix_j + O(|x|^3)\]
for $x \in B^N_+(0,r_1)$ with $r_1 > 0$ small. Here $1 \le i, j, k, l \le n$,
\begin{enumerate}
\item[-] $\tni$ is the second fundamental form on $(M, \hh) \subset (\ox, \bg)$ and $\|\tni\|^2 = \bg^{ik}\bg^{jl}\tni_{ij}\tni_{kl}$;
\item[-] $H$ is the mean curvature on $M$, that is, the average of the diagonal component of $\tni$;
\item[-] $R_{ikjl}[\hh]$ and $R_{iNjN}[\bg]$ are components of the full Riemannian curvature tensors on $(M, \hh)$ and $(\ox, \bg)$, respectively;
\item[-] $R_{ij}[\hh]$ and $R_{NN}[\bg]$ are components of the Ricci tensors on $(M, \hh)$ and $(\ox, \bg)$, respectively.
\end{enumerate}
Commas mean partial derivatives and all the coefficients are computed at $y$.
\end{lemma}
\begin{rmk}\label{rmk-E}
In general, it holds that $|E(\rho)| \le C \rho^{-2\ga}$ in $X$ for some $C > 0$,
so the energy functionals $I^{\ga}(U)$ or $J^{\ga}(U)$ in \eqref{eq-I^ga} and \eqref{eq-J^ga} are well-defined for all $U \in \hx$ and $\ga \in (0, \frac{1}{2})$; see the proof of \cite[Lemma 3.1]{CK}.
Because the coefficient of $x_i x_j x_N$ contains $H_{,ij}$, one of $x_i x_j x_k x_N$ contains $H_{,ijk}$ and so on, if $H = 0$ on $M$, then we have that $|\pa_{\rho} \sqrt{|\bg|}| \le C \rho$ in $X$.
Owing to \eqref{eq-E}, this in turn implies that $|E(\rho)| \le C \rho^{1-2\ga}$ in $X$,
and $I^{\ga}(U)$ and $J^{\ga}(U)$ are well-defined for all $U \in \hx$ and $\ga \in (0,1)$.
\end{rmk}

We also have the following higher order expansion of the metric $\bg$ due to Marques \cite{Ma}.
\begin{lemma}\label{lemma_metric_2}
Given $y \in M$, let $x = (\bx, x_N) \in \mr^N_+$ be the Fermi coordinate on $\ox$ around $y$.
If $\tni = \tni_{;i} = \tni_{;ij} = \tni_{;ijk} = 0$ at $y$, it holds that
\begin{align*}
\sqrt{|\bg|}(x) &= 1 - {1 \over 12} R_{ij;k}[\hh] x_ix_jx_k - {1 \over 2} R_{NN;i}[\bg] x_N^2x_i - {1 \over 6} R_{NN;N}[\bg] x_N^3 \\
&\ - {1 \over 20} \({1 \over 2} R_{ij;kl}[\hh] + {1 \over 9} R_{miqj}[\hh] R_{mkql}[\hh]\) x_ix_jx_kx_l - {1 \over 4} R_{NN;ij}[\bg] x_N^2x_ix_j \\
&\ - {1 \over 6} R_{NN;Ni}[\bg] x_N^3x_i - {1 \over 24} \left[ R_{NN;NN}[\bg] + 2 (R_{iNjN}[\bg])^2 \right] x_N^4 + O(|x|^5)
\end{align*}
and
\begin{align*}
\bg^{ij}(x) &= \delta_{ij} + {1 \over 3} R_{ikjl}[\hh] x_kx_l + R_{iNjN}[\bg] x_N^2 + {1 \over 6} R_{ikjl;m}[\hh] x_kx_lx_m + R_{iNjN;k}[\bg] x_N^2x_k \\
&\ + {1 \over 3} R_{iNjN;N}[\bg] x_N^3 + \({1 \over 20} R_{ikjl;mq}[\hh] + {1 \over 15} R_{iksl}[\hh] R_{jmsq}[\hh]\) x_kx_lx_mx_q \\
&\ + \({1 \over 2} R_{iNjN;kl}[\bg] + {1 \over 3} \textnormal{Sym}_{ij}(R_{iksl}[\hh] R_{sNjN}[\bg])\)x_N^2x_kx_l + {1 \over 3} R_{iNjN;kN}[\bg] x_N^3x_k \\
&\ + {1 \over 12} \(R_{iNjN;NN}[\bg] + 8 R_{iNsN}[\bg] R_{sNjN}[\bg]\) x_N^4 + O(|x|^5)
\end{align*}
for $x \in B^N_+(0,r_1)$. Here $1 \le i, j, k, l, m, q \le n$, semicolons mean covariant derivatives and every tensor is computed at $y$.
\end{lemma}

The next lemmas explain how to choose a good conformal metric $\hh$ on $M$ and to control extrinsic quantities
such as the mean curvature $H$ or the second fundamental form $\tni$ on $M$ with the help of conditions \eqref{eq-R-dec} or \eqref{eq-R-dec-2}.
\begin{lemma}\label{lemma-g-mtr}
Suppose that $(X^N, g^+)$ is an asymptotically hyperbolic manifold with conformal infinity $(M, [\hh])$.
If \eqref{eq-R-dec} holds, then there exist a representative $\hh_0 \in [\hh]$,
the geodesic boundary defining function $\rho_0$ associated to $\hh_0$ and $\bg_0 = \rho^2_0g^+$ such that
\begin{equation}\label{eq-hh_0}
H = 0 \text{ on } M, \quad R_{ij}[\hh_0](y) = 0 \quad \text{and} \quad R_{\rho_0\rho_0}[\bg_0](y) = {1-2n \over 2(n-1)} \|\tni_0(y)\|^2
\end{equation}
for a fixed point $y \in M$, where $\tni_0$ is the second fundamental form on $(M, \hh_0) \subset (\ox, \bg_0)$.
\end{lemma}
\begin{proof}
The proof can be found in \cite[Lemma 2.4]{KMW}.
\end{proof}

The next lemma is a slight generalization of \cite[Lemma 3.2]{KMW} in that the decay condition of $R[g^+]$ is relieved compared to that of \cite{KMW}.
\begin{lemma}\label{lemma-g-mtr-2}
For $n \ge 3$, let $(X^N, g^+)$ be an asymptotically hyperbolic manifold such that the conformal infinity $(M^n, [\hh])$ is umbilic.
If \eqref{eq-R-dec-2} holds, then there exist a representative $\hh_0 \in [\hh]$,
the geodesic boundary defining function $\rho_0$ ($= x_N$ near $M$) associated to $\hh_0$ and the metric $\bg_0 = \rho_0^2 g^+$ such that \eqref{eq-hh_0} is true and
\begin{itemize}
\item[(1)] $\textnormal{Sym}_{ijk} R_{ij;k}[\hh_0](y) = 0$,\
$\textnormal{Sym}_{ijkl}\(R_{ij;kl}[\hh_0] + \dfrac{2}{9} R_{miqj}[\hh_0] R_{mkql}[\hh_0]\)(y) = 0$,
\item[(2)] $\tni = 0$ on $M$, $R_{NN;N}[\bg_0](y) = R_{aN}[\bg_0](y) = 0$,
\item[(3)] $R_{;ii}[\bg_0](y) = -\dfrac{n\|W_0(y)\|^2}{6(n-1)}$,\ $R_{NN;ii}[\bg_0](y) = - \dfrac{\|W_0(y)\|^2}{12(n-1)}$,
\item[(4)] $R_{iNjN}[\bg_0](y) = R_{ij}[\bg_0](y)$,\ $R_{NN;NN}[\bg_0](y) = \dfrac{3}{2n} R_{;NN}[\bg_0](y) - 2 (R_{ij}[\bg_0](y))^2$,
\item[(5)] $R_{iNjN;ij}[\bg_0](y) = \(\dfrac{3-n}{2n}\) R_{;NN}[\bg_0](y) - (R_{ij}[\bg_0](y))^2 - \dfrac{\|W_0(y)\|^2}{12(n-1)}$
\end{itemize}
for a fixed point $y \in M$, where $\|W_0\|$ is the norm of the Weyl tensor $W_0 := W[\hh_0]$ of $(M,\hh_0)$.
\end{lemma}
\begin{proof}
Fix $y \in M$. In view of the previous lemma and the existence of a conformal normal coordinate \cite[Theorem 5.2]{LP}, see also \cite{Ca, Gun},
we may assume that the metric $\hh_0$ on $M$ satisfies \eqref{eq-hh_0} and (1) in the statement.
Then we see from the umbilicity of $M$ and Lemma \ref{lemma_metric_2} that $\tni(y) = 0$ and that \eqref{eq-R-dec-2} is equivalent to
\begin{align*}
&2n \left[ - R_{NN}[\bg_0] - R_{NN;i}[\bg_0] x_i - {1 \over 2} R_{NN;N}[\bg_0] x_N - {1 \over 2} R_{NN;ij}[\bg_0] x_ix_j \right.\\
&\quad \left. - {1 \over 2} R_{NN;Ni}[\bg_0] x_Nx_i + \left\{{1 \over 2} (R_{NN}[\bg_0])^2 - {1 \over 6} R_{NN;NN}[\bg_0] - {1 \over 3} (R_{iNjN}[\bg_0])^2 \right\} x_N^2\right] \\
&\ \times \( 1 + {1 \over 2} R_{NN}[\bg_0] x_N^2 \) \\
&\ + \( R[\bg_0] + R_{,i}[\bg_0] x_i + R_{,N}[\bg_0] x_N + {1 \over 2} R_{,ij}[\bg_0]x_ix_j + R_{,iN}[\bg_0]x_ix_N + {1 \over 2} R_{,NN}[\bg_0] x_N^2 \) \\
&= o(x_N^2) + O(|x|^3)
\end{align*}
in the Fermi coordinate on $\ox$ around $y$. All tensors here are evaluated at $y$.
Hence the coefficient of $1, x_N, x_i^2$ and $x_N^2$ in the left-hand side should be $0$, which implies
\begin{equation}\label{eq-R-rel}
0 = 2n R_{NN}[\bg_0] = R[\bg_0],\quad n R_{NN;N}[\bg_0] = R_{;N}[\bg_0],\quad 2n R_{NN;ii}[\bg_0] = R_{;ii}[\bg_0]
\end{equation}
and
\begin{equation}\label{eq-R-rel-2}
2n \left[R_{NN;NN}[\bg_0] + 2(R_{iNjN}[\bg_0])^2 \right] = 3 R_{;NN}[\bg_0]
\end{equation}
at $y $. Note that $R_{,aa}[\bg_0] = R_{;aa}[\bg_0]$ at $y$ for each $a = 1, \cdots, N$.

By the Codazzi equation $R_{ijkN}[\bg_0] = \tni_{ik;j} - \tni_{jk;i}$ and the second Bianchi identity, it holds that $R_{iN}[\bg_0] = 0$ and $R_{,N}[\bg_0] = 2 R_{NN;N}[\bg_0]$ at $y$.
This with the second equality in \eqref{eq-R-rel} gives $R_{NN;N}[\bg_0] = 0$ at $y$, proving (2).

Furthermore, the Gauss-Codazzi equation and the proof of \cite[Theorem 5.1]{LP} lead to
\[R_{;ii}[\bg_0] = 2R_{NN;ii}[\bg_0] - R_{;ii}[\hh] = 2R_{NN;ii}[\bg_0] - {\|W_0\|^2 \over 6}\]
at $y$. Hence, together with the third equality in \eqref{eq-R-rel}, we deduce (3).

We also have that
\[R_{ij}[\bg_0] = R_{ikjk}[\bg_0] + R_{iNjN}[\bg_0] = R_{ij}[\hh_0] + R_{iNjN}[\bg_0] = R_{iNjN}[\bg_0]\]
at $y$. Therefore (4) follows from \eqref{eq-R-rel-2}.

Finally, combining the contracted second Bianchi identity $R_{;NN}[\bg_0] = 2R_{aN;aN}[\bg_0]$, the Ricci identity
\[R_{iN;iN}[\bg_0] - R_{iN;Ni}[\bg_0] = R_{aiiN}[\bg_0] R_{aN}[\bg_0] + R_{aNiN}[\bg_0] R_{ia}[\bg_0] = (R_{ij}[\bg_0])^2\]
and
\[R_{iN;Ni}[\bg_0] = - R_{NjjN;ii}[\bg_0] - R_{NjNi;ji}[\bg_0] = - \dfrac{\|W_0\|^2}{12(n-1)} - R_{iNjN;ij}[\bg_0]\]
at $y$, we arrive at (5). The proof is finished.
\end{proof}

\section{An existence criterion}\label{sec-ps-cri}
\subsection{Constraint set and accompanying quantities} \label{subsec-cons}
Throughout this section, we always assume that $f$ is a smooth function on $M$, somewhere positive.
One way to search solutions for \eqref{eq-ps} is to examine the existence of positive solutions to a more general class of problems
\begin{equation}\label{eq-sub}
P^{\ga}_{\hh} u = f |u|^{\beta-1} u \quad \text{on } (M^n, \hh)
\end{equation}
for $\beta \in (1, 2^*]$. Thanks to Proposition \ref{prop-GQ-ext} and Lemma \ref{lemma-GQ-ext}, it can be interpreted as
\begin{equation}\label{eq-sub-ext}
\begin{cases}
-\textnormal{div}_{\bg} (\rho^{1-2\ga} \nabla U) + E(\rho)U = 0 &\text{in } (X, \bg),\\
U = u &\text{on } M,\\
\pa^{\ga}_{\nu} U = f |u|^{\beta-1} u &\text{on } M,
\end{cases}
\end{equation}
where the last equation should be modified adequately if $\ga = \frac{1}{2}$.

Since $f$ is positive at some point on $M$, the {\it constraint set}
\begin{equation}\label{eq-mcbf}
\mcc_{\beta, f} = \left\{U \in \hx: U = u \text{ on } M,\, \int_M f |u|^{\beta+1} dv_{\hh} = 1 \right\}
\end{equation}
is nonempty. If we set
\begin{equation}\label{eq-wtla}
\oth^{\ga}(\beta, f) = \inf_{U \in \mcc_{\beta, f}(M)} I^{\ga}(U),
\quad \wtth^{\ga}(\beta, f) = \inf_{U \in \mcc_{\beta, f}(M)} J^{\ga}(U)
\end{equation}
and
\begin{equation}\label{eq-labf}
\Theta^{\ga}(\beta, f) = \inf \left\{\int_M u P^{\ga}_{\hh} u\, dv_{\hh}: u \in H^{\ga}(M), \int_M f |u|^{\beta+1} dv_{\hh} = 1 \right\}
\end{equation}
where $I^{\ga}$ and $J^{\ga}$ are functionals defined in \eqref{eq-I^ga} and \eqref{eq-J^ga},
then we have the following relation.
\begin{lemma}\label{lemma-la}
Suppose that \eqref{eq-eig} is valid, $H = 0$ for $\ga \in (\frac{1}{2}, 1)$, $\Lambda^{\ga}(M, [\hh]) > 0$ and $\beta \in (1, 2^*]$. Then it holds that
\[\Theta^{\ga}(\beta, f) = \oth^{\ga}(\beta, f) = \wtth^{\ga}(\beta, f) \ge 0.\]
\end{lemma}
\begin{proof}
By the condition $\Lambda^{\ga}(M, [\hh]) > 0$, \cite[Theorem 1.1]{Cs} and the density argument,
\[0 \le \Theta^{\ga}(\beta, f) \le \min \left\{ \oth^{\ga}(\beta, f), \wtth^{\ga}(\beta, f) \right\}.\]

Suppose that $\{u_m\}_{m \in \mn} \subset C^{\infty}(M)$ is a minimizing sequence of $\Theta^{\ga}(\beta, f)$.
By virtue of \eqref{eq-eig}, for each $u_m$, the eigenvalue problem
\[-\Delta_{g^+} v - s(n-s) v = 0 \quad \text{in } X, \quad s := {n \over 2} + \ga\]
has a solution of the form
\begin{equation}\label{eq-v_m-form}
v_m = F \rho^{n-s} + G \rho^s, \quad F, G \in C^{\infty}(\ox), \quad F|_{\rho = 0} = u_m
\end{equation}
where $\rho$ is the geodesic defining function associated to $(M, \hh)$; see for instance \cite{MM, JS, GZ}.
In \cite{CG}, it is proved that $U_m := \rho^{-(n-s)} v_m \in \hx$ is a solution of \eqref{eq-ext-1} with $u = u_m$ and $\pa^{\ga}_{\nu} U_m = P^{\ga}_{\hh} u_m$ on $M$.
Thus, by putting $U_m$ in \eqref{eq-ext-1}, we observe
\[\oth^{\ga}(\beta, f) \le I^{\ga}(U_m) = \int_M u_m P^{\ga}_{\hh} u_m dv_{\hh} \to \Theta^{\ga}(\beta, f).\]
Similarly, we have that $\wtth^{\ga}(\beta, f) \le \Theta^{\ga}(\beta, f)$. This finishes the proof.
\end{proof}
If $H = 0$ on $M$ and $U \in \mcc_{\beta, f}$ achieves the infimum $\oth^{\ga}(\beta, f)$, then it solves
\begin{equation}\label{eq-sub-ext-2}
\begin{cases}
-\text{div}_{\bg} (\rho^{1-2\ga} \nabla U) + E(\rho)U = 0 &\text{in } (X, \bg),\\
\pa^{\ga}_{\nu} U = \oth^{\ga}(\beta, f) f |U|^{\beta-1} U &\text{on } M
\end{cases}
\end{equation}
in the weak sense. Since $|U|$ also attains $\oth^{\ga}(\beta, f)$, we may assume that $0 \le U \ne 0$ in $X$.
Then Remark \ref{rmk-reg-1} below implies that $U$ is in fact positive in $\ox$ and $\oth^{\ga}(\beta, f) > 0$; refer also to the Hopf lemma in \cite[Theorem 3.5, Corollary 3.6]{GQ}.
Therefore a constant multiple of $U$ gives a positive solution of \eqref{eq-sub-ext}.
In the next subsection, we shall provide a criterion which guarantees the existence of a minimizer $U \in \mcc_{\beta, f}$.

\subsection{Subcritical approximation}\label{subsec-sub}
The main goal of this subsection is to show
\begin{prop}\label{prop-cpt}
Suppose that \eqref{eq-eig} holds, $H = 0$ for $\ga \in (\frac{1}{2}, 1)$, $\Lambda^{\ga}(M, [\hh]) > 0$ and $f$ is a smooth function positive somewhere on $M$.
For any $U \in \mcc_{2^*, f}$, it is valid that
\begin{equation}\label{eq-cpt}
\(\max_{x \in M} f(x)\)^{n-2\ga \over n} \Theta^{\ga}(2^*, f) \le \ola^{\ga}(\mh^N, [\hh_c])
\end{equation}
where $\Theta^{\ga}(2^*, f)$ and $\ola^{\ga}(\mh^N, [\hh_c])$ are quantities defined in \eqref{eq-labf} and \eqref{eq-H^n-ene}, respectively.
Moreover, if the strict inequality in \eqref{eq-cpt} holds, there exists a function $U_{2^*} \in \mcc_{2^*, f}$ attaining $\oth^{\ga}(2^*, f)$.
It can be chosen to be positive on $\ox$, and so it gives a weak solution to \eqref{eq-ps-ext}.
\end{prop}
\begin{rmk}
(1) Escobar and Schoen \cite{ES}, and Escobar \cite{Es} gave the proof of the above result for $\ga = 1$ and $\frac{1}{2}$, which is rather sketchy.
Aubin \cite{Au} also proved it provided that $\ga =1$ and $f$ is positive everywhere on $M$.

\medskip \noindent
(2) Setting $f = 1$ recovers the solvability criterion for the fractional Yamabe problem, which appeared firstly in Gonz\'alez and Qing \cite[Theorem 1.4]{GQ}.
Our argument is a bit more complicated than that of \cite{GQ}, because we allow the situation that $f$ attains a negative value.
See the proof of Lemmas \ref{lemma-wtu-bdd} and \ref{lemma-wtth-conv-2} below.

\medskip \noindent
(3) Suppose that
\begin{equation}\label{eq-str}
- \infty < \ola^{\ga}(X, [\hh]) < \ola^{\ga}(\mh^N, [\hh_c])
\end{equation}
holds, which is the case when one of the conditions (a)-(d) in Theorem \ref{thm-Y} is true.
Then there exists a small number $\ep > 0$ which depends only on the underlying manifold $(X^{n+1}, g^+)$, its boundary $(M^n, [\hh])$ and $\ga \in (0,1)$
such that if $f$ is a positive function satisfying $\sup_M f \le (1+\ep) \inf_M f$,
then the strict inequality in \eqref{eq-cpt} holds.
To verify it, one may estimate $\wtth^{\ga}(2^*,f)$ with the constant function $(\int_M f dv_{\hh})^{-1/(2^*+1)} \in \mcc_{2^*,f}$.
This argument was given by Aubin \cite{Au} for $\ga = 1$.
\end{rmk}

The former part of Proposition \ref{prop-cpt} can be deduced immediately.
Assume that there exists a number $r_2 > 0$ small enough so that for each fixed point $y \in M$, the Fermi coordinate around $y$ is well-defined in its $(2r_2)$-geodesic neighborhood in $\ox$.
Let also $\chi$ be a smooth radial cut-off function in $\overline{\mr^N_+}$ which satisfies
\begin{equation}\label{eq-chi}
0 \le \chi \le 1 \quad \text{on } \overline{\mr^N_+} \quad \text{and} \quad \chi = \begin{cases}
1 &\text{in } \overline{B^N_+(0,1)},\\
0 &\text{outside } \overline{B^N_+(0,2)}.
\end{cases}
\end{equation}
\begin{proof}[Proof of Proposition \ref{prop-cpt}: Derivation of \eqref{eq-cpt}]
We will adapt the proof of \cite[Theorem 1.4]{GQ}.
Choose $y \in M$ such that $f(y) = \max_{x \in M} f(x) > 0$ and consider the Fermi coordinate on $\ox$ around $y$, identifying $y \in M$ with $0 \in \mr^n$.
If $\chi_{r_2} := \chi(\cdot/r_2) \in C^{\infty}_c(\overline{\mr^N_+})$, then for small $\ep > 0$,
\[0 < \mu_0 := \int_{B^n(0,2r_2)} f (\chi_{r_2} w_{\ep,0})^{2^*+1} \sqrt{|\hh|} d\bx = f(y) \int_{\mr^n} w_{1,0}^{2^*+1} d\bx + o(1)\]
where $o(1) \to 0$ as $\ep \to 0$. Hence $U_{\ep} := \mu_0^{-(2^*+1)} \chi_{r_2} W_{\ep, 0} \in \mcc_{2^*,f}$, and by \eqref{eq-H^n-ene},
\begin{align*}
\(\max_{x \in M} f(x)\)^{n-2\ga \over n} \wtth^{\ga}(2^*, f)
&\le f(y)^{n-2\ga \over n} \mu_0^{-{n-2\ga \over n}} \kappa_{\ga} \int_{B^N_+(0,2r_2)} x_N^{1-2\ga} |\nabla (\chi_{r_2} W_{\ep,0})|^2 \sqrt{|\bg|} dx + o(1)\\
&= \(\int_{\mr^n} w_{1,0}^{2^*+1} d\bx\)^{2\ga \over n} + o(1) = \ola^{\ga}(\mh^N, [\hh_c]) + o(1).
\end{align*}
Taking $\ep \to 0$ and employing Lemma \ref{lemma-la}, we obtain \eqref{eq-cpt}.
\end{proof}
On the other hand, we need several preliminary lemmas to prove the latter part of Proposition \ref{prop-cpt}.

The next result follows from the standard variational argument with the compactness of the trace operator $\hx \hookrightarrow L^{\beta}(M)$ for $\beta \in (2, 2^*+1)$
and the strong maximum principle given in Remark \ref{rmk-reg-1}. We omit the proof.
\begin{lemma}\label{lemma-wtu}
Suppose that \eqref{eq-eig} holds, $H = 0$ for $\ga \in (\frac{1}{2}, 1)$, $\Lambda^{\ga}(M, [\hh]) > 0$ and $f$ is a smooth function positive somewhere on $M$.
For all $\beta \in (1, 2^*)$, there exists a positive minimizer $\wtu_{\beta} \in \mcc_{\beta, f}(M)$ of $\wtth^{\ga}(\beta, f)$, which solves
\begin{equation}\label{eq-sub-ext-1}
\begin{cases}
-\textnormal{div}_{\bg^*} \((\rho^*)^{1-2\ga} \nabla U\) = 0 &\text{in } (X, \bg^*),\\
\pa^{\ga}_{\nu} U = \wtth^{\ga}(\beta, f) f |U|^{\beta-1} U - Q^{\ga}_{\hh} U &\text{on } M;
\end{cases}
\end{equation}
refer to \eqref{eq-mcbf} and \eqref{eq-wtla}.
\end{lemma}

The following lemma shows that $\wtth^{\ga}(\beta, f)$ is upper semi-continuous from the left at $\beta = 2^*$.
Note that finiteness of $\wtth^{\ga}(2^*, f)$ is guaranteed by \eqref{eq-cpt}.
\begin{lemma}\label{lemma-wtth-conv-1}
It holds that
\begin{equation}\label{eq-wtth-conv-1}
\limsup_{\beta \to 2^*-} \wtth^{\ga}(\beta, f) \le \wtth^{\ga}(2^*, f).
\end{equation}
\end{lemma}
\begin{proof}
Let $\mcd$ be the function space introduced in \eqref{eq-mcd}. Observe that for each $U \in \mcd$ such that $\int_M f |U|^{\beta+1} dv_{\hh} > 0$, we have
\[\limsup_{\beta \to 2^*-} \wtth^{\ga}(\beta, f)
\le \limsup_{\beta \to 2^*-} {J^{\ga}(U) \over \(\int_M f |U|^{\beta+1} dv_{\hh}\)^{2 \over \beta+1}}
\le {J^{\ga}(U) \over \(\int_M f |U|^{2^*+1} dv_{\hh}\)^{n-2\ga \over n}}.\]
Taking the infimum over $U \in \mcd$, we deduce the desired inequality \eqref{eq-wtth-conv-1}.
\end{proof}

We also need a uniform regularity result on the family of functions $\wtu_{\beta}$.
\begin{lemma}\label{lemma-wtu-bdd}
Suppose that the assumptions in Lemma \ref{lemma-wtu} hold.
For each $\beta \in (1, 2^*)$, let $\wtu_{\beta} \in \mcc_{\beta, f}(M)$ be a positive minimizer of $\wtth^{\ga}(\beta, f)$.
If $\Lambda^{\ga}(M, [\hh]) > 0$, then there exists a constant $C > 0$ and small $\vep > 0$ such that
\begin{equation}\label{eq-wtu-bdd}
\sup_{\beta \in [2^*-\vep, 2^*)} \|\wtu_{\beta}\|_{H^{1,2}(X;\rho^{1-2\ga})} \le C
\quad \text{and} \quad
\sup_{\beta \in [2^*-\vep, 2^*)} \|\wtu_{\beta}\|_{C^{\alpha}(\mck)} \le C
\end{equation}
for some $\alpha \in (0,1)$ and $\mck := \{y \in M: f(y) \le 0\}$.
\end{lemma}
\begin{proof}
Let $\vep > 0$ be a sufficiently small number.
We see from \eqref{eq-Y-const-2} and \eqref{eq-wtth-conv-1} that
\begin{equation}\label{eq-coer}
\Lambda^{\ga}(M,[\hh]) \(\int_M \wtu_{\beta}^{2^*+1} dv_{\hh}\)^{n-2\ga \over n} \le J^{\ga}(\wtu_{\beta})
= \wtth^{\ga}(\beta,f) \le 2 \wtth^{\ga}(2^*, f)
\end{equation}
for $\beta \in [2^*-\vep, 2^*)$. Thus
\[\int_X (\rho^*)^{1-2\ga} |\nabla \wtu_{\beta}|_{\bg^*}^2 dv_{\bg^*} \le 2 \wtth^{\ga}(2^*, f) + \|Q^{\ga}_{\hh}\|_{L^{\infty}(M)} \int_M \wtu_{\beta}^2 dv_{\hh}\]
and the right-hand side is uniformly bounded in $\beta \in [2^*-\vep, 2^*)$.
On the other hand, \cite[Lemma 3.1]{CK} implies that the norm
\[\(\int_X (\rho^*)^{1-2\ga} |\nabla U|_{\bg^*}^2 dv_{\bg^*} + \int_M U^2 dv_{\hh}\)^{\frac{1}{2}} \quad \text{for } U \in \hx\]
is equivalent to the standard $\hx$-norm. As a result, we establish the first inequality in \eqref{eq-wtu-bdd}.

\medskip
To deduce the second inequality in \eqref{eq-wtu-bdd}, it suffices to verify that there exists an $\eta$-neighborhood $B_{\eta}(\mck) \subset M$ of $\mck$ such that
\begin{equation}\label{eq-wtu-bdd-2}
\sup_{\beta \in [2^*-\vep, 2^*)} \|\wtu_{\beta}\|_{L^{\infty}(B_{\eta}(\mck))} \le C.
\end{equation}
Then together with the De Giorgi-Nash-Moser estimate stated in Lemma \ref{lemma-reg-1} (2), one can get the $C^{\alpha}(\mck)$-uniform estimate for $\{\wtu_{\beta}\}_{\beta \in [2^*-\vep, 2^*)}$ for some $\alpha \in (0,1)$.

We will apply the blow-up argument close to the proof of \cite[Theorem 1.4]{GQ}.
Suppose that there exist sequences $\beta_m \to 2^*$, $\wtu_m := \wtu_{\beta_m}$ and $y_m \in B_{\eta}(\mck)$ such that
\[M_m := \wtu_m(y_m) = \|\wtu_m\|_{L^{\infty}(B_{\eta}(\mck))} \to \infty,\quad \wtth^{\ga}(\beta_m,f) \to \wtth_0 \le \wtth^{\ga}(2^*,f)\]
and $y_m \to y_0 \in \overline{B_{\eta}(\mck)}$ as $m \to \infty$. Take a Fermi coordinate system around $y_0$, identified with $0 \in \mr^N_+$, and define
\[\wtv_m(x) = M_m^{-1} \wtu_m (\delta_m \bx + y_m, \delta_m x_N) \quad \text{for } x = (\bx, x_N) \in B^N_+(0,2r_2/\delta_m)\]
where $\delta_m := M_m^{-\frac{\beta_m-1}{2\ga}}$ and $r_2 > 0$ is a small number.
Let also $\chi_{r_2/\delta_m} = \chi(\delta_m \cdot/r_2)$ where $\chi \in C^{\infty}_c(\overline{\mr^N_+})$ is an arbitrary function such that \eqref{eq-chi} holds.
Then $V_m := \chi_{r_2/\delta_m} \wtv_m$ solves
\[\begin{cases}
-\text{div}_{\bg^*_m} \((\rho^*_m)^{1-2\ga} \nabla V_m\) = 0 &\text{in } B^N_+(0,r_2/\delta_m)\\
\pa^{\ga}_{\nu} V_m = \wtth^{\ga}(\beta_m, f) f_m V_m^{\beta_m} - (Q^{\ga}_{\hh})_m \delta_m^{2\ga} V_m &\text{on } B^n(0,r_2/\delta_m)
\end{cases} \]
where
\[\bg^*_m(\bx, x_N) := \bg^*(\delta_m \bx + y_m, \delta_m x_N), \quad \rho^*_m(\bx, x_N) := \rho^*(\delta_m \bx + y_m, \delta_m x_N),\]
and $f_m$ and $(Q^{\ga}_{\hh})_m$ are similarly defined. Also,
\begin{equation}\label{eq-V_m}
\begin{aligned}
\int_{B^N_+(0,2r_2/\delta_m)} x_N^{1-2\ga} |\nabla V_m|^2 dx
&\le C \int_{B^N_+(0,2r_2/\delta_m)} x_N^{1-2\ga} \(|\nabla \wtv_m|^2 + |\nabla \chi_{r_2/\delta_m}|^2 \wtv_m^2\) dx \\
&\le C M_m^{{(n-2\ga)(\beta_m-1) \over 2\ga}-2} \|\wtu_m\|_{\hx}^2 \le C_0
\end{aligned}
\end{equation}
for some constant $C_0 > 0$ where the exponent of $M_m$ is always negative since $\beta_m < 2^*$.
Therefore $V_m \to V_0$ strongly in $C^{\alpha'}_{\text{loc}} (\overline{\mr^N_+})$ and weakly $\dot{H}^{1,2}(\mr^N_+; x_N^{1-2\ga})$ for some nonzero bounded function $V_0$ and $\alpha' \in (0,1)$.
It is easy to check that $V_0$ is a solution to
\begin{equation}\label{eq-limit}
\begin{cases}
-\text{div}(x_N^{1-2\ga} \nabla V_0) = 0 &\text{in } \mr^N_+,\\
\pa^{\ga}_{\nu} V_0 = \wtth_0 f(y_0) V_0^{n+2\ga \over n-2\ga} &\text{on } \mr^n.
\end{cases}
\end{equation}
If $\wtth_0 = 0$ or $f(y_0) \le 0$, a contradiction immediately arises since it should hold that $V_0 = 0$ in $\mr^N_+$.
Suppose that $\wtth_0 > 0$ and $f(y_0) > 0$.
For any $\delta > 0$, one can select small $\eta > 0$ so that $f(y) \le \delta$ for any $y \in B_{\eta}(\mck)$.
By the classification theorem \cite[Theorem 1.8]{JLX} of Eq. \eqref{eq-limit},
we know that $V_1 := (\wtth_0 f(y_0))^{n-2\ga \over 4\ga} V_0$ is the bubble $W_{\lambda, 0}$ for some $\lambda > 0$. Consequently,
\begin{align*}
C_0 \ge \|V_0\|_{\dot{H}^{1,2}(\mr^N_+;x_N^{1-2\ga})} &= (\wtth_0 f(y_0))^{-{n-2\ga \over 4\ga}} \|W_{\lambda,0}\|_{\dot{H}^{1,2}(\mr^N_+;x_N^{1-2\ga})} \\
&\ge \(\wtth^{\ga}(2^*,f) \delta\)^{-{n-2\ga \over 4\ga}} \|W_{1,0}\|_{\dot{H}^{1,2}(\mr^N_+;x_N^{1-2\ga})},
\end{align*}
which is a contradiction to \eqref{eq-V_m} provided $\delta > 0$ small enough.
Hence \eqref{eq-wtu-bdd-2} is true.
\end{proof}

If $\Lambda^{\ga}(M, [\hh]) > 0$, we are able to improve Lemma \ref{lemma-wtth-conv-1} by showing the $\wtth^{\ga}(\beta, f)$ is continuous from the left at $\beta = 2^*$.
Unlike the fractional Yamabe problem, it is less clear here due to the negative part of $f$.
\begin{lemma}\label{lemma-wtth-conv-2}
Under the assumption in Lemma \ref{lemma-wtu}, we have
\[\lim_{\beta \to 2^*-} \wtth^{\ga}(\beta, f) = \wtth^{\ga}(2^*, f).\]
\end{lemma}
\begin{proof}
Let $\wtu_{\beta} \in \mcc_{\beta, f}$ be the positive minimizer of $\wtth^{\ga}(\beta, f)$.
Then we infer from H\"older's inequality and Lemma \ref{lemma-wtu-bdd} that for any given $\delta > 0$, there exists a small number $\vep > 0$ such that
\[\int_M f \wtu_{\beta}^{2^*+1} dv_{\hh} \ge \int_M f \wtu_{\beta}^{\beta+1} dv_{\hh} - \delta = 1 - \delta\]
for all $\beta \in [2^*-\vep, 2^*)$. Therefore we have
\[\wtth^{\ga}(2^*, f) \le {J^{\ga}(U) \over (\int_M f \wtu_{\beta}^{2^*+1} dv_{\hh})^{n-2\ga \over n}} \le (1-\delta)^{-{n-2\ga \over n}} \wtth^{\ga}(\beta, f)\]
for $\beta \in [2^*-\vep, 2^*)$. Taking $\delta \to 0$, we prove the assertion.
\end{proof}

We are now ready to conclude the proof of Proposition \ref{prop-cpt}.
\begin{proof}[Completion of the proof of Proposition \ref{prop-cpt}]
By \eqref{eq-wtu-bdd}, there exists a nonnegative function $\wtu_{2^*} \in \hx$ such that $\wtu_{\beta} \rightharpoonup \wtu_{2^*}$ weakly in $\hx$.
Thanks to Lemma \ref{lemma-reg-1} (2), it is not hard to see that $\wtu_{2^*}$ is H\"older continuous on $\ox$ and solves \eqref{eq-sub-ext-1} with $\beta = 2^*$.
Besides, by the strong maximum principle in Remark \ref{rmk-reg-1}, we have that $\wtu_{2^*} > 0$ on $\ox$ unless it is trivial.

We claim that $\wtu_{2^*}$ is nonzero provided that the strict inequality in \eqref{eq-cpt} holds.
As in \cite[Proposition 2.5]{JX}, one can prove that for any $\ep > 0$, there exists $A(\ep) > 0$ such that
\[\(\int_M |U|^{2^*+1} dv_{\hh}\)^{2 \over 2^*+1} \le (1 + \ep) S_{n,\ga} \int_X (\rho^*)^{1-2\ga} |\nabla U|^2_{\bg^*} dv_{\bg^*} + A(\ep) \int_M U^2 dv_{\hh}\]
for any $U \in \hx$. Since $J^{\ga}(\wtu_{\beta}) = \wtth^{\ga}(\beta, f)$, it follows that
\begin{equation}\label{eq-est-1}
\(\int_M \wtu_{\beta}^{2^*+1} dv_{\hh}\)^{2 \over 2^*+1} \le (1 + \ep) S_{n,\ga} \kappa_{\ga}^{-1} \wtth^{\ga}(\beta, f) + A'(\ep) \int_M \wtu_{\beta}^2 dv_{\hh}
\end{equation}
where $A'(\ep) := (S_{n,\ga} + \ep) \kappa_{\ga}^{-1} \|Q^{\ga}_{\hh}\|_{L^{\infty}(M)} + A(\ep)$.
Meanwhile, we get from $\wtu_{\beta} \in \mcc_{\beta, f}$ and H\"older's inequality that
\begin{equation}\label{eq-est-2}
1 \le \(\max_{x \in M} f(x)\) |M|^{2^*-\beta \over \beta+1} \(\int_M \wtu_{\beta}^{2^*+1} dv_{\hh}\)^{\beta+1 \over 2^*+1}.
\end{equation}
Thus, putting \eqref{eq-est-1}, \eqref{eq-est-2}, \eqref{eq-H^n-ene} and Lemma \ref{lemma-wtth-conv-2} together, we obtain
\[1 - (1+2\ep) \(\max_{x \in M} f(x)\)^{n-2\ga \over n} \wtth^{\ga}(2^*, f) \(\ola^{\ga}(\mh^N, [\hh_c])\)^{-1} \le C \int_M \wtu_{\beta}^2 dv_{\hh}\]
for $\beta$ close to $2^*$. Now the left-hand side is positive if $\ep > 0$ is chosen small enough. Hence the $L^2(M)$-norm of $\wtu_{2^*}$, or $\wtu_{2^*}$ itself, is nonzero.

\medskip
Thanks to the assumption $\Lambda^{\ga}(M,[\hh]) > 0$ and \eqref{eq-coer}, we obtain
\begin{equation}\label{eq-wtu_2*}
\wtth^{\ga}(2^*,f) \int_M f \wtu_{2^*}^{2^*+1} dv_{\hh} = J^{\ga}(\wtu_{2^*}) \ge 0
\end{equation}
by testing $\wtu_{2^*}$ in \eqref{eq-sub-ext-1} with $\beta = 2^*$.
Thus Lemma \ref{lemma-la} yields that $\wtth^{\ga}(2^*,f) > 0$.
Notice that we cannot have that $\wtth^{\ga}(2^*,f) = 0$, since we would get $\wtu_{2^*} = 0$ on $M$ if it is so.
By reasoning in the same way, we obtain $\mu_1 := \int_M f \wtu_{2^*}^{2^*+1} dv_{\hh} > 0$ as well.
Using the lower semi-continuity of $J^{\ga}$ and Lemma \ref{lemma-wtth-conv-2}, we see
\[\wtth^{\ga}(2^*,f) \mu_1 = J^{\ga}(\wtu_{2^*}) \le \liminf_{\beta \to 2^*-} J^{\ga}(\wtu_{\beta}) = \lim_{\beta \to 2^*-} \wtth^{\ga}(\beta,f) = \wtth^{\ga}(2^*,f),\]
so $\mu_1 \in (0,1]$. Now if we set $V = \mu_1^{-1/(2^*+1)} \wtu_{2^*} \in \mcc_{2^*,f}$, we deduce from \eqref{eq-wtu_2*} that
\[\wtth^{\ga}(2^*,f) \le J^{\ga}(V) = \mu_1^{-{n-2\ga \over n}} J^{\ga}(\wtu_{2^*}) = \mu_1^{2\ga \over n} \wtth^{\ga}(2^*,f).\]
Therefore $\mu_1 = 1$ and $\wtu_{2^*} \in \mcc_{2^*,f}$ is a minimizer of $\wtth^{\ga}(2^*,f)$.
Let $U_{2^*} = (\rho^*/\rho)^{n-2\ga \over 2} \wtu_{2^*} > 0$ on $\ox$.
Then it is an element in $\mcc_{2^*,f}$ which attains $\oth^{\ga}(2^*, f)$ and solves \eqref{eq-sub-ext-2}.
In view of the discussion at the end of Subsection \ref{subsec-cons}, the proof is completed.
\end{proof}

\section{Existence results}\label{sec-ps-ex}
\subsection{Proof of Theorems \ref{thm-ps} and \ref{thm-ps-end}}
In light of Proposition \ref{prop-cpt}, we just need to verify that the strict inequality in \eqref{eq-cpt} holds in each situation.
As in the previous section, we assume that for each fixed $y \in M$, the Fermi coordinate around $y$ is well-defined in its $(2r_2)$-geodesic neighborhood on $\ox$.

\medskip \noindent \textsc{Condition (A).}
Let $\chi \in C^{\infty}_c(\overline{\mr^N_+})$ be a cut-off function satisfying \eqref{eq-chi} and $\chi_{r_2} = \chi(\cdot/r_2)$.
In \cite[Proposition 2.5]{KMW}, it is proved that for sufficiently small $\ep > 0$,
\begin{multline*}
I^{\ga}(\chi_{r_2} W_{\ep, 0}) \le \kappa_{\ga} \int_{\mr^N_+} x_N^{1-2\ga} |\nabla W_{1,0}|^2 dx \\
+ \ep \kappa_{\ga} H(y) \left[{2n^2-2n+1-4\ga^2 \over 2(1-2\ga)}\right] \int_{\mr^N_+} x_N^{2-2\ga} |\nabla W_{1,0}|^2 dx + o(\ep)
\end{multline*}
where $I^{\ga}$ is the functional given in \eqref{eq-I^ga} and $W_{\ep,0}$ is the bubble defined in terms of Eq. \eqref{eq-bubble}. On the other hand, if $y \in M$ is the maximum point of $f$, then
\[{f(y) \over \int_M f (\chi_{r_2} w_{\ep, 0})^{2^*+1} dv_{\hh}} = \(\int_{\mr^n} w_{1, 0}^{2^*+1} d\bx\)^{-1} + O(\ep^2).\]
Combining these two estimates and using Lemma \ref{lemma-la}, we conclude that
\begin{equation}\label{eq-str-2}
\begin{aligned}
\(\max_{x \in M} f(x)\)^{n-2\ga \over n} \Theta^{\ga}(2^*, f)
&= \left[{f(y) \over \int_M f (\chi_{r_2} w_{\ep, 0})^{2^*+1} dv_{\hh}}\right]^{n-2\ga \over n} I^{\ga}(\chi_{r_2} W_{\ep, 0}) \\
&< \ola^{\ga}(\mh^N, [\hh_c]).
\end{aligned}
\end{equation}
From Proposition \ref{prop-cpt}, we obtain a weak solution to \eqref{eq-ps-ext}.

\medskip \noindent \textsc{Condition (B).}
We pick $\hh_0 \in [\hh]$ satisfying \eqref{eq-hh_0} and define
\begin{equation}\label{eq-psi_1}
\Psi_{1\ep}(x) := C_1 \text{II}_{ij}(y) x_i x_j x_N r^{-1} \pa_r W_{\ep,0}(x) \quad \text{for } x \in \mr^N_+
\end{equation}
where $1 \le i, j \le n$, $r := |\bx|$ and $C_1 \in \mr$.
Then the computation in the proof of \cite[Proposition 2.8]{KMW} shows that if we select $C_1 \in \mr$ appropriately, then we get
\begin{multline*}
I^{\ga}(\chi_{r_2} (W_{\ep,0} + \Psi_{1\ep})) \le \kappa_{\ga} \int_{\mr^N_+} x_N^{1-2\ga} |\nabla W_{1,0}|^2 dx \\
- \ep^2 \kappa_{\ga} \|\text{II}(y)\|^2 \mcm_1(n,\ga) \int_{\mr^N_+} x_N^{3-2\ga} |\nabla W_{1,0}|^2 dx + o(\ep^2)
\end{multline*}
with
\begin{multline*}
\mcm_1(n,\ga) := \({1+\ga \over 3}\) \left[{3n^2 + n(16\ga^2-22)+20(1-\ga^2) \over 8n(n-1)(1-\ga^2)} \right.\\
\left. + {16(n-1)(1-\ga^2) \over n(3n^2+n(2-8\ga^2)+4\ga^2-4)}\right] > 0 \quad \text{for } n \ge 4,\ \ga \in (0,1).
\end{multline*}
Since $y \in M$ is the maximum point of the function $f$, we have
\begin{multline}\label{eq-f}
\int_M f (\chi_{r_2} (w_{\ep, 0} + 0))^{2^*+1} dv_{\hh_0} = f(y) \int_{\mr^n} w_1^{2^*+1} d\bx \\
+ \ep^2 \({\Delta f(y) \over 2n}\) \int_{\mr^n} |\bx|^2 w_1^{2^*+1} d\bx + O(\ep^3)
\end{multline}
where the number 0 in the integrand of the left-hand side emphasizes that $\Psi_{1\ep}$ vanishes on $M$.
As remarked in the proof of \cite[Theorem 3.3]{Es}, we still have \eqref{eq-cond-1} with the same value of $c^1_{n,\ga}$ even though $\hh$ is replaced with $\hh_0$.
Therefore, choosing
\begin{equation}\label{eq-c_1}
\begin{aligned}
c^1_{n,\ga} &= \left[{2n(n-2) \over n-2\ga}\right] \left[{\int_{\mr^N_+} x_N^{3-2\ga} |\nabla W_{1,0}|^2 dx
\over \int_{\mr^N_+} x_N^{1-2\ga} |\nabla W_{1,0}|^2 dx}\right] \mcm_1(n, \ga) \\
&= \left[{16n(n-1)(1-\ga)\ga \over (n-2\ga)(n-2+2\ga)(n-2-2\ga)}\right] \mcm_1(n, \ga),
\end{aligned}
\end{equation}
where the second equality can be computed as in \cite{GQ, KMW}, we obtain
\begin{align*}
\(\max_{x \in M} f(x)\)^{n-2\ga \over n} \Theta^{\ga}(2^*, f)
&\le \ola^{\ga}(\mh^N, [\hh_c]) + \ep^2 c^{11}_{n,\ga} \({-\Delta_{\hh} f(y) \over f(y)} - c^1_{n,\ga} \|\text{II}\|^2(y)\) + o(\ep^2) \\
&< \ola^{\ga}(\mh^N, [\hh_c])
\end{align*}
for some $c^{11}_{n,\ga} > 0$. Proposition \ref{prop-cpt} implies the existence of a solution to \eqref{eq-ps-ext}.

\medskip \noindent \textsc{Condition (B$^{\prime}$).}
Putting \eqref{eq-Ma} and the estimate
\[\int_M f (\chi_{r_2} (w_{\ep, 0} + C_1' \text{II}_{ij}x_ix_j \pa_r w_{\ep,0}))^3 dv_{\hh_0} = f(y) \int_{\mr^n} w_1^3 d\bx + O(\ep^2)\]
together, we find
\[\(\max_{x \in M} f(x)\)^{2 \over 3} \Theta^{\ga}(2, f) \le \ola^{\ga}(\mh^4, [\hh_c]) - C \ep^2 |\log{\ep}| + O(\ep^2) < \ola^{\ga}(\mh^4, [\hh_c])\]
for some $C > 0$. Therefore there exists a solution to \eqref{eq-ps-ext}.

\medskip \noindent \textsc{Conditions (C) and (C$^{\prime}$).}
As before, we pick $\hh_0 \in [\hh]$ satisfying \eqref{eq-hh_0}. Then, under condition (c), we have
\begin{multline*}
I^{\ga}(\chi_{r_2} W_{\ep, 0}) \le \kappa_{\ga} \int_{\mr^N_+} x_N^{1-2\ga} |\nabla W_{1,0}|^2 dx \\
+ \ep^3 \kappa_{\ga} R_{NN;N}(y) \left[{4n^2-12n+9-4\ga^2 \over 24n(3-2\ga)}\right]
\int_{\mr^N_+} x_N^{4-2\ga} |\nabla W_1|^2 dx + o(\ep^3)
\end{multline*}
as computed in \cite[Proposition 3.4]{KMW}. Since $-\Delta_{\hh} f(y) = 0$ and $w_{\ep, 0}$ is radially symmetric, \eqref{eq-f} gives
\[\int_M f (\chi_{r_2} w_{\ep, 0})^{2^*+1} dv_{\hh_0} = f(y) \int_{\mr^n} w_1^{2^*+1} d\bx + O(\ep^4).\]
Hence \eqref{eq-str-2} holds provided that $R_{NN;N}(y) < 0$ and \eqref{eq-ps-ext} has a solution.
A similar calculation can be conducted when condition (c$^{\prime}$) holds.

\medskip \noindent \textsc{Condition (D).}
Select $\hh_0 \in [\hh]$ satisfying all the conditions imposed in Lemma \ref{lemma-g-mtr-2}
and write $\hh_0 = u \hh$ for some positive function $u$ on $M$.
Set also a function
\begin{equation}\label{eq-psi_2}
\Psi_{2\ep} = C_2 R_{iNjN}[\bg](y) x_i x_j x_N^2 r^{-1} \pa_r W_{\ep,0} \quad \text{in } \mr^N_+
\end{equation}
where $1 \le i, j \le n$, $r = |\bx|$ and $C_2 \in \mr$. From the proof of \cite[Proposition 3.7]{KMW}, we see
\begin{multline*}
I^{\ga}(\chi_{r_2} (W_{\ep,0} + \Psi_{2\ep})) \le \kappa_{\ga} \int_{\mr^N_+} x_N^{1-2\ga} |\nabla W_{1,0}|^2 dx \\
- \ep^4 \kappa_{\ga} \left[\|W_0\|^2 \mcm_{21}(n,\ga) + (R_{ij}[\bg_0])^2 \mcm_{22}(n,\ga)\right] \int_{\mr^N_+} x_N^{5-2\ga} |\nabla W_{1,0}|^2 dx + o(\ep^4)
\end{multline*}
for an optimally chosen number $C_2 \in \mr$. Here, the metric $\bg$ is replaced with $\bg_0$, $W_0 = W[\hh_0]$,
\[\mcm_{21}(n,\ga) := \frac{15n^4 - 120n^3 + 20n^2(17-2\ga^2) - 80n(5-2\ga^2) + 48(4-5\ga^2+\ga^4)} {7680n(n-1)(n-3)(1+\ga)(1-\ga)(2-\ga)}\]
and
\[\mcm_{22}(n, \ga) := \frac{25n^3 - 20n^2(9-\ga^2) + 100n(4-\ga^2) - 16(4-\ga^2)^2} {5n(n-4-2\ga)(n-4+2\ga)(5n^2-4n(1+\ga^2)-8(4-\ga^2))}.\]
The constants $\mcm_{21}(n,\ga)$ and $\mcm_{22}(n,\ga)$ are positive for any $n > 4 + 2\ga$ and $\ga \in (0,1)$.
Because $y \in M$ is the maximum point of the function $f$ at which $\Delta_{\hh} f = 0$, we have
\[\Delta_{\hh_0} f(y) = u^{-1} \Delta_{\hh} f(y) = 0 \quad \text{and } \quad \pa_{ij} f(y) = 0 \quad \text{for each } 1 \le i, j \le n,\]
where the latter assertion can be checked by mathematical induction on $n$.
Furthermore, one obtains using Lemma \ref{lemma-g-mtr-2} (1) that
\begin{multline}\label{eq-f-2}
\int_M f (\chi_{r_2} w_{\ep, 0})^{2^*+1} dv_{\hh_0} = f(y) \int_{\mr^n} w_1^{2^*+1} d\bx \\
+ \ep^4 \left[{(-\Delta)^2 f(y) \over 8n(n+2)}\right] \int_{\mr^n} |\bx|^4 w_1^{2^*+1} d\bx + O(\ep^5).
\end{multline}
Meanwhile, the function $u$ can be assumed to satisfy $u(y) = 1$ and $u_{,i}(y) = 0$; see Lee-Parker \cite[Section 5]{LP}. It follows that
\begin{equation}\label{eq-c2-1}
(-\Delta_{\hh_0})^2 f(y) = u(y)^{-2} \left[ (-\Delta_{\hh})^2 f + (n-2) \pa_{ij}u\, \pa_{ij}f \right](y) = (-\Delta_{\hh})^2 f(y).
\end{equation}
Moreover, it is a well-known fact that the $(1,3)$-Weyl tensor is invariant under the conformal transformation. Hence
\begin{equation}\label{eq-c2-2}
\begin{aligned}
\|W_0\|^2 &= \hh_0^{ij}\hh_0^{kl}\hh_0^{pq}\hh_0^{rs} (W_0)_{ikpr} (W_0)_{jlqs} = \hh_0^{kl}\hh_0^{rs} (W_0)^j_{\phantom{j}kpr} (W_0)^p_{\phantom{p}sjl} \\
&= u^{-2} \hh^{kl}\hh^{rs} W^j_{\phantom{j}kpr} W^p_{\phantom{p}sjl} = \|W\|^2
\end{aligned}
\end{equation}
at $y$, where the indices $i, j, k, l, p, q, r$ and $s$ range from 1 to $n$.
By \eqref{eq-c2-1} and \eqref{eq-c2-2}, assumption \eqref{eq-cond-2} should be still valid with the same value of $c_{n,\ga}^2 > 0$, even after we substitute $\hh_0$ for $\hh$.
As a consequence, our selection
\begin{equation}\label{eq-c_2}
\begin{aligned}
c^2_{n,\ga} &= \left[{8n(n-2)(n-4) \over n-2\ga}\right] \left[{\int_{\mr^N_+} x_N^{5-2\ga} |\nabla W_{1,0}|^2 dx
\over \int_{\mr^N_+} x_N^{1-2\ga} |\nabla W_{1,0}|^2 dx}\right] \mcm_{21}(n, \ga) \\
&= \left[{1024 (n-3)(n-1)n (2-\ga)(1-\ga)\ga \over 3(n-2\ga) (n-2-2\ga)(n-2+2\ga) (n-4-2\ga)(n-4+2\ga)}\right] \mcm_{21}(n, \ga),
\end{aligned}
\end{equation}
where the second equality comes from the arguments in \cite{GQ, KMW}, allows us to deduce
\begin{align*}
&\ \(\max_{x \in M} f(x)\)^{n-2\ga \over n} \Theta^{\ga}(2^*, f) \\
&\le \ola^{\ga}(\mh^N, [\hh_c]) - \ep^4 c^{21}_{n,\ga} \({(-\Delta_{\hh})^2 f(y) \over f(y)} + c^2_{n,\ga} \|W\|^2(y) + c^{22}_{n,\ga} (R_{ij}[\bg_0])^2\) + o(\ep^4) \\
&< \ola^{\ga}(\mh^N, [\hh_c])
\end{align*}
for some $c^{21}_{n,\ga},\, c^{22}_{n,\ga} > 0$.

\medskip \noindent \textsc{Condition (D$^{\prime}$).}
The desired inequality \eqref{eq-str-2} follows from \eqref{eq-Al} and \eqref{eq-f-2}.

\subsection{Proof of Theorem \ref{thm-ps-PE}}
We shall give a brief sketch of the proof of Theorem \ref{thm-ps-PE} under (E).
Let $\Phi_{\ep, r_2}$ be a Schoen-type test function constructed in (4.9) (with $\varrho_0 = r_2$) of Kim et al. \cite{KMW},
which is equal to the bubble $W_{\ep, 0}$ in $X \cap B^N(0,r_2)$ and a constant multiple of the Green's function $G(x,0)$ in $X \setminus B^N(0,2r_2)$. It is nonnegative in $X$ and satisfies
\[I^{\ga}(\Phi_{\ep, r_2}) \le \kappa_{\ga} \int_{\mr^N_+} x_N^{1-2\ga} |\nabla W_{1,0}|^2 dx - c_{n,\ga}^3 A \ep^{n-2\ga} + o(\ep^{n-2\ga})\]
for some $c_{n,\ga}^3 > 0$; see the proof of \cite[Proposition 4.5]{KMW}. The vanishing condition on $f$ implies
\[\int_M f \Phi_{\ep, r_2}^{2^*+1} dv_{\hh} \ge \int_{M \cap B^n(0,r_2)} f w_{\ep,0}^{2^*+1} d\bx
= f(y) \int_{\mr^n} w_{1,0}^{2^*+1} d\bx + o(\ep^{n-2\ga}).\]
Hence a strict inequality in \eqref{eq-cpt} holds.
A similar argument works for the case (E$^{\prime}$).

\section{Regularity of solutions to the prescribed fractional scalar curvature problems}\label{sec-reg}
\subsection{Local regularity results}
Suppose that $n > 2\ga$ and $\ga \in (0,1) \setminus \{\frac{1}{2}\}$.
Also, we assume the following conditions hold:
\begin{enumerate}
\item[(R1)] $A = (A_{ab})_{a,b=1}^N \in C^2(\overline{B_r}, \mr^{N \times N})$ satisfies that $A_{iN} = 0$ in $B_r$ for $i = 1, \cdots, n$
    and the uniform ellipticity condition $\Lambda_1 |\xi|^2 \le A(x) \xi \cdot \xi \le \Lambda_2 |\xi|^2$ holds
    for all points $\xi \in \mr^N$, $x \in B_r$ and some constants $0 < \Lambda_1 \le \Lambda_2$;
\item[(R2)] $F = (F_1, \cdots, F_n, F_N) \in L^{q_1}(B_r; x_N^{1-2\ga})$ and $B, G \in L^{q_2}(B_r; x_N^{1-2\ga})$
    for some $q_1 > n-2\ga+2$ and $q_2 > \frac{n-2\ga+2}{2}$;
\item[(R3)] $t \in L^{p_2}(\pa' B_r)$ for $p_2 > \frac{n}{2\ga}$.
\end{enumerate}
In this subsection, we present several regularity results for degenerate elliptic equations having the form
\begin{equation}\label{eq-reg}
\begin{cases}
-\text{div} (x_N^{1-2\ga} A \nabla U) + x_N^{1-2\ga} B U = x_N^{1-2\ga} G + \text{div} (x_N^{1-2\ga} F) &\text{in } B_r,\\
\pa^{\ga}_{\nu, F} U = sU + t &\text{on } \pa' B_r
\end{cases}
\end{equation}
where $B_r := B^N_+(0,r)$ and $\pa' B_r := B^n(0,r)$ for a fixed radius $r > 0$ and
\[\pa^{\ga}_{\nu, F} U := \kappa_{\ga} \lim_{x_N \to 0+} x_N^{1-2\ga} \(F_N - {\pa U \over \pa x_N} \) = \pa^{\ga}_{\nu} U + \kappa_{\ga} \lim_{x_N \to 0+} x_N^{1-2\ga} F_N.\]
We say that $U \in H^{1,2}(B_r; x_N^{1-2\ga})$ is a {\it weak solution} of \eqref{eq-reg} if it is a solution in the sense of a distribution.

\medskip
The first result concerns integrability and H\"older regularity of weak solutions to \eqref{eq-reg}.
\begin{lemma}\label{lemma-reg-1}
Suppose that functions $A$, $B$, $F$, $G$ and $t$ satisfy \textnormal{(R1)}, \textnormal{(R2)} and \textnormal{(R3)}, and $U \in H^{1,2}(B_r; x_N^{1-2\ga})$ is a weak solution of \eqref{eq-reg}.

\medskip \noindent (1) There is a small number $\delta_0 > 0$ relying only on $n$, $\ga$ and $r$ such that if $\|s\|_{L^{\frac{n}{2\ga}}(\pa' B_r)} < \delta_0$, then
\begin{multline*}
\|U\|_{H^{1,2}(B_{r/2}; x_N^{1-2\ga})} + \|U\|_{L^{2n(n-2\ga+2) \over (n-2\ga)^2}(\pa' B_{r/2})} \\
\le C \(\|U\|_{L^2(B_r; x_N^{1-2\ga})} + \|F\|_{L^{q_1}(B_r; x_N^{1-2\ga})} + \|G\|_{L^{q_2}(B_r; x_N^{1-2\ga})} + \|t\|_{L^{p_2}(\pa' B_r)}\)
\end{multline*}
where $C > 0$ depends only on $n$, $\ga$, $r$, $\Lambda_1$, $\Lambda_2$ and $\|B\|_{L^{q_2}(B_r; x_N^{1-2\ga})}$.

\medskip \noindent (2) Assume that $s \in L^{p_1}(\pa' B_r)$ for $p_1 > \frac{n}{2\ga}$.
Then $U \in C^{\alpha}(\overline{B_{r/2}})$ for some $\alpha \in (0,1)$ and
\[\|U\|_{C^{\alpha}(\overline{B_{r/2}})} \le C \(\|U\|_{L^2(B_r; x_N^{1-2\ga})} + \|F\|_{L^{q_1}(B_r; x_N^{1-2\ga})} + \|G\|_{L^{q_2}(B_r; x_N^{1-2\ga})} + \|t\|_{L^{p_2}(\pa' B_r)}\)\]
where $C > 0$ depends only on $n$, $\ga$, $r$, $\Lambda_1$, $\Lambda_2$, $\|B\|_{L^{q_2}(B_r; x_N^{1-2\ga})}$ and $\|s\|_{L^{p_1}(\pa' B_r)}$.
\end{lemma}
\begin{proof}
One can justify (1) by following the argument in \cite[Proposition 2.3, Lemma 2.8]{JLX} but using the optimal Sobolev inequality \cite[Lemma 3.1]{FF}.

The assertion (2) follows from the standard argument involving the Moser iteration technique.
Refer to \cite[Proposition 2.6]{JLX} and \cite[Proposition 3]{FF} which cover the case that $A$ is the identity and $F = 0$ in $B_r$.
\end{proof}
\begin{rmk}\label{rmk-reg-1}
On the course of the proof of Lemma \ref{lemma-reg-1} (2), one gets the {\it weak Harnack inequality}:
Suppose that the hypotheses of Lemma \ref{lemma-reg-1} (2) hold and $U \ge 0$ in $B_r$. Then there is a number $q_0 \ge 1$ such that
\[\inf_{\overline{B_{r/4}}} U + \|F\|_{L^{q_1}(B_r; x_N^{1-2\ga})} + \|G\|_{L^{q_2}(B_r; x_N^{1-2\ga})} + \|t\|_{L^{p_2}(\pa' B_r)}
\ge C\|U\|_{L^{q_0}(B_{r/2}; x_N^{1-2\ga})} \]
where $C > 0$ depends only on $n$, $\ga$, $r$, $\Lambda_1$, $\Lambda_2$, $\|B\|_{L^{q_2}(B_r; x_N^{1-2\ga})}$ and $\|s\|_{L^{p_1}(\pa' B_r)}$.
For its derivation, see \cite[Proposition 2.6]{JLX} and \cite[Proposition 2]{FF} which treat the situation that $A$ is the identity and $F = 0$ in $B_r$.

From this result, one obtains the {\it strong maximum principle}: If $F, G = 0$ in $B_r$, $t = 0$ on $\pa' B_r$ and $U \ge 0$ attains 0 at some point in $B_r$, then $U = 0$ in $B_r$.
\end{rmk}

Secondly, we establish regularity of the derivatives of weak solutions to \eqref{eq-reg} in the $\bx$-variables.
Let us introduce a terminology:
For any function $V$ in $\mr^N_+$ and $h \in  \mr^n$ with $|h|$ small, define
\[D_h V(\bx, x_N) = {V(\bx+h, x_N) - V(\bx, x_N) \over |h|} \quad \text{for } (\bx, x_N) \in \mr^N_+.\]
We call $D_h V$ a {\it difference quotient} of $V$.
\begin{lemma}\label{lemma-reg-2}
Suppose that $A$ satisfies \textnormal{(R1)} and $U \in H^{1,2}(B_r; x_N^{1-2\ga})$ is a weak solution of \eqref{eq-reg}.
Additionally, we assume that
\begin{enumerate}
\item[(R11)] there exist a function $A_{NN}' \in C^2(\overline{B_r})$ and a number $\Lambda_1 \le c \le \Lambda_2$ such that $A_{NN}(x) = c + A_{NN}'(x) x_N$ on $\overline{B_r}$. In particular,
    $|D_h A_{NN}| + |D_h (\pa_1 A_{NN})| + \cdots + |D_h (\pa_n A_{NN})| \le C x_N$ on $\overline{B_r}$;
\item[(R21)] $F_a, \nabla_{\bx} F_a := (\pa_i F_a)_{i=1}^n \in C^{\alpha'}(\overline{B_r})$ for $\alpha' \in (0,1)$ and $B, \nabla_{\bx} B, G, \nabla_{\bx} G \in L^{\infty}(B_r)$;
\item[(R22)] $\nabla_{\bx}^2 F_a = (\pa_{ij} F_a)_{i,j = 1}^n \in L^{q_1}(B_r; x_N^{1-2\ga})$
    and $\nabla_{\bx}^2 B, \nabla_{\bx}^2 G \in L^{q_2}(B_r; x_N^{1-2\ga})$ for $q_1 > n-2\ga+2$ and $q_2 > \frac{n-2\ga+2}{2}$;
\item[(R31)] $s, \nabla_{\bx} s, \nabla_{\bx}^2 s \in L^{p_1}(\pa' B_r)$
    and $t, \nabla_{\bx} t, \nabla_{\bx}^2 t \in L^{p_2}(\pa' B_r)$ for $p_1,\, p_2 > \frac{p}{2\ga}$.
\end{enumerate}
Then there exists $0 < r' < r$ such that $\nabla_{\bx} U, \nabla_{\bx}^2 U \in C^{\alpha} (\overline{B_{r'}})$ for some $\alpha \in (0,1)$.
\end{lemma}
\begin{proof}
We shall show that $\nabla_{\bx} U$ is H\"older continuous on $\overline{B_{r/12}}$. For each $h \in \mr^n$ with $|h|$ small, the corresponding difference quotient $D_h U$ of $U$ solves
\begin{equation}\label{eq-reg-21}
\begin{cases}
-\text{div} (x_N^{1-2\ga} A \nabla (D_h U)) + x_N^{1-2\ga} B (D_h U) = x_N^{1-2\ga} \wtg + \text{div} (x_N^{1-2\ga} \wtf) &\text{in } B_{r/2},\\
\pa^{\ga}_{\nu, \wtf} (D_h U) = s (D_h U) + [(D_h s) U_h + D_h t] &\text{on } \pa' B_{r/2}
\end{cases}
\end{equation}
where $U_h(\bx, x_N) := U(\bx + h, x_N)$,
\begin{align}
\wtf &:= (D_h F_i + (D_h A_{ij}) (\pa_j U_h), D_h F_N + (D_h A_{NN}) (\pa_N U_h)), \label{eq-wtf}\\
\wtg &:= D_hG - (D_hB) U_h. \nonumber
\end{align}
By (R11) and Lemma \ref{lemma-reg-1} (2), we have
\begin{align*}
\|D_h U\|_{C^{\alpha}(\overline{B_{r/12}})} &\le C \(\|\nabla U\|_{L^2(B_r; x_N^{1-2\ga})} + \|\nabla_{\bx} U\|_{L^{q_1}(B_{r/6}; x_N^{1-2\ga})} + \|x_N \pa_N U\|_{L^{q_1}(B_{r/6}; x_N^{1-2\ga})} \right. \\
&\qquad + \|\nabla_{\bx} F\|_{L^{q_1}(B_r; x_N^{1-2\ga})} + \|\nabla_{\bx} G\|_{L^{q_2}(B_r; x_N^{1-2\ga})} + \|\nabla_{\bx} B\|_{L^{q_2}(B_r; x_N^{1-2\ga})} \\
&\qquad \left. + \|\nabla_{\bx} s\|_{L^{p_2}(\pa' B_r)} + \|\nabla_{\bx} t\|_{L^{p_2}(\pa' B_r)}\)
\end{align*}
where $C > 0$ depends only on $n$, $\ga$, $r$, $A$, $B$, $\|s\|_{L^{p_1}(\pa' B_r)}$ and $\|U\|_{L^{\infty}(B_r)}$.
In view of (R21) and (R31), it is sufficient to check that $x_N \pa_N U \in L^{\infty}(B_{r/4})$ and $\nabla_{\bx} U \in L^q(B_{r/6}; x_N^{1-2\ga})$ for any $q > 1$.

We apply the rescaling argument to prove the first claim. For a fixed point $x_0 = (\bx_0, x_{N0}) \in B_{r/4}$ and any element $x \in B^N((0,1), \frac{1}{2})$,
we set $x' = (\bx', x_N') = (\bx_0 + x_{N0} \bx, x_{N0} x_N) \in B^N(x_0, \frac{x_{0N}}{4})$, $\widehat{U}(x) = U(x')$, $\widehat{A}(x) = A(x')$, etc.
Then $\widehat{U}$ is a solution to the equation
\[-\text{div}\(x_N^{1-2\ga} \widehat{A}\, \nabla \widehat{U}\) + x_{0N}^2 x_N^{1-2\ga} \widehat{B} \widehat{U} = x_{0N}^2 x_N^{1-2\ga} \widehat{G} + x_{0N} \text{div}\(x_N^{1-2\ga} \widehat{F}\)\]
in $B^N((0,1), \frac{1}{2})$. It is uniformly elliptic, so an application of \cite[Theorem 8.32]{GT} and (R21) give
\begin{equation}\label{eq-reg-23}
\begin{aligned}
|(x_N \pa_NU)(x_0)| &\le \left\| \pa_N \widehat{U} \right\|_{C^{\alpha'} (\overline{B^N((0,1), \frac{1}{4})})} \\
&\le C \( \|U\|_{L^{\infty}(B_{r/2})} + \|F\|_{C^{\alpha'}(\overline{B_r})} + \|G\|_{L^{\infty}(B_r)} \)
\end{aligned}
\end{equation}
where $C > 0$ depends only on $n$, $r$, $A$ and $\|B\|_{L^{\infty}(B_r)}$.

To examine the second claim, let us select
\[k = \begin{cases}
\begin{aligned}
\|\nabla_{\bx} F\|_{L^{q_1}(B_r; x_N^{1-2\ga})} + \|\nabla_{\bx} G\|_{L^{q_2}(B_r; x_N^{1-2\ga})} + \|\nabla_{\bx} B\|_{L^{q_2}(B_r; x_N^{1-2\ga})} \\
+ \|\nabla_{\bx} s\|_{L^{p_2}(\pa' B_r)} + \|\nabla_{\bx} t\|_{L^{p_2}(\pa' B_r)}
\end{aligned} &\text{if it is nonzero},\\
\text{any positive number} &\text{otherwise}.
\end{cases}\]
In the latter case, we send $k \to 0$ at the last step. For a fixed $K > 0$ and $m \ge 0$, we define
\[V_h = (D_h U)_+ + k = \max\{D_h U, 0\} + k,\quad V_{h,K} = \min\{V_h, K\},\quad Z_{h,K,m} = V_{h,K}^{m \over 2} V_h.\]
Let $\chi \in C^{\infty}(\overline{\mr^N_+})$ be a cut-off function satisfying \eqref{eq-chi} and $|\nabla \chi| \le C \chi$ in $\mr^N_+$.
Setting $\chi_{r/6} = \chi(6\cdot/r)$, we test \eqref{eq-reg-21} with $\Xi_{h,K,m} := \chi_{r/6}^2 (V_{h,K}^m V_h - k^{m+1})$.
Then a bit of calculation exploiting H\"older's inequality, Young's inequality, the Sobolev inequality, the Sobolev trace inequality and (R11) shows
\begin{equation}\label{eq-reg-22}
\begin{aligned}
&\ \int_{B_r} x_N^{1-2\ga} |\nabla (\chi_{r/6} Z_{h,K,m})|^2 dx \\
&\le C \left[ \int_{B_r} x_N^{1-2\ga} \chi_{r/6}^2 \(V_{h,K}^{m-1} V_h |\nabla V_{h,K}|^2 + V_{h,K}^m |\nabla V_h|^2 \) dx + \int_{B_r} x_N^{1-2\ga} (\chi_{r/6} Z_{h,K,m})^2 dx \right] \\
&\le C \left[\int_{B_r} x_N^{1-2\ga} \(|\nabla_{\bx} U_h| + |x_N \pa_N U_h|\) |\nabla \Xi_{h,K,m}|\, dx + \int_{B_r} x_N^{1-2\ga} (\chi_{r/6} Z_{h,K,m})^2 dx \right]
\end{aligned}
\end{equation}
for $C > 0$ depending only on $n$, $\ga$, $r$, $A$, $B$, $\|s\|_{L^{p_1}(\pa' B_r)}$, $\|U\|_{L^{\infty}(B_r)}$ and $m$.
The boundary integrals appearing here can be controlled as in the proof of \cite[Proposition 2.6]{JLX}.
Because of the first claim, the first integral in the rightmost side of \eqref{eq-reg-22} is bounded by
\begin{multline*}
\int_{B_r} x_N^{1-2\ga} \(|\nabla_{\bx} U_h| + \|x_N \pa_N U\|_{L^{\infty}(B_{r/2})}\) |\nabla \Xi_{h,K,m}|\, dx \le \vep \int_{B_r} x_N^{1-2\ga} \chi_{r/6}^2 V_{h,K}^m |\nabla V_h|^2 dx \\
+ C \left[ \int_{B_r} x_N^{1-2\ga} \chi_{r/6}^2 V_{h,K}^m (|\nabla_{\bx} U_h|^2 + \|x_N \pa_N U\|_{L^{\infty}(B_{r/2})}^2) dx + \int_{B_r} x_N^{1-2\ga} (\chi_{r/6} Z_{h,K,m})^2 dx \right]
\end{multline*}
for a small $\vep > 0$. As a consequence, by taking $K \to \infty$ and applying the Sobolev inequality, we reach at
\begin{multline*}
\|V_h\|_{L^{(m+2) \cdot ({n-2\ga+2 \over n-2\ga})}(B_{r/6}; x_N^{1-2\ga})}^{m+2} \\
\le C \left[ \|V_h\|_{L^{m+2}(B_{r/3}; x_N^{1-2\ga})}^{m+2} + \int_{B_{r/3}} x_N^{1-2\ga} V_h^m (|\nabla_{\bx} U_h|^2 + \|x_N \pa_N U\|_{L^{\infty}(B_{r/2})}^2) dx \right]
\end{multline*}
whenever the right-hand side is finite.

Since $\nabla U \in L^2(B_r; x_N^{1-2\ga})$, the $L^2(B_{r/3}; x_N^{1-2\ga})$-norm of $V_h$ is uniformly bounded in $h \in \mr^n$ with small $|h| > 0$.
By taking $m = 0$ in the above estimate, we deduce that the $L^{2(n-2\ga+2) \over n-2\ga}(B_{r/6}; x_N^{1-2\ga})$-norm of $V_h$ is uniformly bounded in $h$,
which implies that $(\nabla_{\bx} U)_+ \in L^{2(n-2\ga+2) \over n-2\ga}(B_{r/6}; x_N^{1-2\ga})$.
The same argument applied to $(D_h U)_- + k$ where $(D_h U)_- := \max\{-D_h U, 0\}$ gives us
that $(\nabla_{\bx} U)_- \in L^{2(n-2\ga+2) \over n-2\ga}(B_{r/6}; x_N^{1-2\ga})$ as well.
Repeating this procedure many times and adjusting $r > 0$, we discover
\[\|\nabla_{\bx} U \|_{L^q(B_{r/6}; x_N^{1-2\ga})} \le C \( \|\nabla_{\bx} U \|_{L^2(B_r; x_N^{1-2\ga})} + \|x_N \pa_N U\|_{L^{\infty}(B_{r/2})} + k \)\]
for any $q > 1$ where $C > 0$ depends only on $n$, $\ga$, $r$, $A$, $B$, $\|s\|_{L^{p_1}(\pa' B_r)}$, $\|U\|_{L^{\infty}(B_r)}$ and $q$.
This justifies the second claim.
The fact that $\nabla_{\bx} U \in H^{1,2}(B_{r/6}; x_N^{1-2\ga})$ follows from \eqref{eq-reg-22} as a by-product.

\medskip
In a similar way, one can prove that $\nabla_{\bx}^2 U$ is H\"older continuous on $\overline{B_{r/144}}$.
Because of (R22), the only nontrivial part is to check that $\|x_N \pa_{Ni} U\|_{L^{\infty}(B_{r/24})} < \infty$ for $i = 1, \cdots, n$.
By elliptic regularity, there holds that
\begin{align*}
|(x_N \pa_{Ni} U)(x_0)| &\le \left\| \pa_N \widehat{\pa_i U} \right\|_{C^{\min\{\alpha, \alpha'\}}(B^N((0,1), \frac{1}{4}))} \\
&\le C \(\|\pa_i U\|_{C^{\alpha}(\overline{B_{r/12}})} + \left\| \widehat{x_N \pa_N U} \right\|_{C^{\alpha'}(B^N((0,1), \frac{1}{2}))} \right. \\
&\qquad \left. + \|\nabla_{\bx} F\|_{C^{\alpha'}(\overline{B_r})} + \|\nabla_{\bx} G\|_{L^{\infty}(B_r)} + \|\nabla_{\bx} B\|_{L^{\infty}(B_r)}\)
\end{align*}
for each $x_0 \in B_{r/24}$, where $\widehat{\pa_i U}$ and $\widehat{x_N \pa_N U}$ are defined as before
and the constant $C > 0$ depends only on $n, r, A$ and $\|U\|_{L^{\infty}(B_{r/2})}$.
Furthermore, given $x, y \in B^N((0,1), \frac{1}{2}))$, if we denote $x' = (\bx_0 + x_{N0}\bx, x_{N0}x_N)$ and $y' = (\bx_0 + x_{N0}\bar{y}, x_{N0}y_N)$, then $x', y' \in B^N(x_0, \frac{x_{0N}}{2})$ and
\begin{align*}
\left| \widehat{x_N \pa_N U}(x) - \widehat{x_N \pa_N U}(y) \right| &= \left|(x_{0N}x_N) (\pa_N U)(x') - (x_{0N}y_N) (\pa_N U)(y') \right| \\
&\le x_{0N} \left[ |x_N - y_N| |(\pa_N U)(x')| + |y_N| |(\pa_N U)(x') - (\pa_N U)(y')| \right] \\
&\le 2 \|x_N \pa_N U\|_{L^{\infty}(B_{r/4})} |x_N - y_N| + {3 \over 2} \left[ \pa_N \widehat{U} \right]_{C^{\alpha'}(\overline{B_{r/2}})} |x - y|^{\alpha'} \\
&\le C \( \|x_N \pa_N U\|_{L^{\infty}(B_{r/4})} + \left\| \pa_N \widehat{U} \right\|_{C^{\alpha'}(\overline{B_{r/2}})} \) |x - y|^{\alpha'}.
\end{align*}
Thus we see from \eqref{eq-reg-23} that
\[\left\| \widehat{x_N \pa_N U} \right\|_{C^{\alpha'}(B^N((0,1), \frac{1}{2}))} \le C \(\|U\|_{L^{\infty}(B_{r/2})} + \|F\|_{C^{\alpha'}(\overline{B_r})} + \|G\|_{L^{\infty}(B_r)}\).\]
This proves the assertion and concludes the proof.
\end{proof}
\noindent Modifying the above argument slightly, one gets the following result.
\begin{cor}\label{cor-reg-2}
Assume that $A$ satisfies \textnormal{(R1)} and
\begin{enumerate}
\item[(R12)] $A \in C^{\infty}(B_r, \mr^{N \times N})$ and $A_{NN}(x) = c + A_{NN}'(x) x_N$ in $B_r$
    for a function $A_{NN}' \in C^{\infty}(B_r)$ and a number $\Lambda_1 \le c \le \Lambda_2$;
\item[(R23)] $\nabla_{\bx}^m F_a \in C^{\alpha'}(B_r)$ and $\nabla_{\bx}^m B, \nabla_{\bx}^m G \in L^{\infty}(B_r)$ for some $\alpha' \in (0,1)$ and all $m \in \mn$.
\end{enumerate}
Moreover, suppose that $U \in H^{1,2}(B_r; x_N^{1-2\ga}) \cap L^{\infty}(B_r)$ is a positive function in $B_r$ and weakly solves
\begin{equation}\label{eq-reg-2}
\begin{cases}
-\textnormal{div} (x_N^{1-2\ga} A \nabla U) + x_N^{1-2\ga} B U = x_N^{1-2\ga} G + \textnormal{div} (x_N^{1-2\ga} F) &\text{in } B_r,\\
\pa^{\ga}_{\nu,F} U = fU^{\beta} &\text{on } \pa' B_r
\end{cases}
\end{equation}
for some $f \in C^{\infty}(\overline{B_r})$ and $\beta \in (1, 2^*]$.
Then, for any $m \in \mn$, there exist $\alpha \in (0,1)$ and $0 < r' < r$ such that $\nabla_{\bx}^m U \in C^{\alpha}(\overline{B_{r'}})$.
\end{cor}
\begin{proof}
We can argue as in the proof of the previous lemma to show that $\nabla_{\bx} U \in C^{\alpha}(B_{r/12})$.
The only difference is the way to deal with the boundary integral in deducing \eqref{eq-reg-22}.
We easily observe that
\[\pa^{\ga}_{\nu, \wtf} (D_h U) = (D_h f) U^{\beta}_h + f (D_h U^{\beta}) \quad \text{on } \pa' B_{r/2}\]
where $\wtf$ is the function defined in \eqref{eq-wtf}. In addition, if we redefine
\[k = \begin{cases}
\begin{aligned}
\|\nabla_{\bx} F\|_{L^{q_1}(B_r; x_N^{1-2\ga})} + \|\nabla_{\bx} G\|_{L^{q_2}(B_r; x_N^{1-2\ga})} + \|\nabla_{\bx} B\|_{L^{q_2}(B_r; x_N^{1-2\ga})} \\
+ \|\nabla_{\bx} f\|_{L^{p_2}(\pa' B_r)}
\end{aligned} &\text{if it is nonzero},\\
\text{any positive number} &\text{otherwise}
\end{cases} \]
and set $\Xi_{h, M, m}$ as before, then we get $|D_h U^{\beta}| \le C\|U\|_{L^{\infty}(\pa' B_r)}^{\beta-1} |D_h U|$ in $\pa' B_{r/2}$ so that
\begin{multline*}
\int_{\pa' B_r} (|D_h f| U^{\beta}_h + |f| |D_h U^{\beta}|)\, \Xi_{h,K,m} d\bx \\
\le C \|U\|_{L^{\infty}(\pa' B_r)}^{\beta-1} \int_{\pa' B_r} \(\|U\|_{L^{\infty}(\pa' B_r)} {|D_h f| \over k} + \|f\|_{L^{\infty}(\pa' B_r)} \) (\chi_{r/6} Z_{h,K,m})^2 d\bx
\end{multline*}
where $C > 0$ depends only on $\beta$. Now, one can derive \eqref{eq-reg-22} having this estimate in hand and following the proof of \cite[Proposition 2.6]{JLX}.

\medskip
H\"older continuity of higher order derivatives $\nabla_{\bx}^m U$ for $m \in \mn$ can be achieved by iteration of this argument.
The positivity of $U$ guarantees that $U^{\beta-m}$ is bounded away from 0 for any $\beta \in (1, 2^*]$, so $|D_h U^{\beta-m}| \le C |D_h U|$ in $B_r$ for some $C > 0$.
\end{proof}

We next deduce H\"older regularity of the weighted normal derivative $x_N^{1-2\ga} \pa_N U$ of a weak solution $U$ to \eqref{eq-reg}
from suitable regularity of the solution, its tangential derivatives $\nabla_{\bx} U, \nabla_{\bx}^2 U$ and the coefficients $A$, $B$, $F$, $G$, $s$, $t$.
\begin{lemma}\label{lemma-reg-3}
Suppose that $U \in H^{1,2}(B_r; x_N^{1-2\ga})$ is a weak solution of \eqref{eq-reg} where the matrix $A$ satisfies \textnormal{(R1)}.
Furthermore, assume that $U, \nabla_{\bx} U, \nabla_{\bx}^2 U \in C^{\alpha}(\overline{B_r})$,
\begin{enumerate}
\item[(R24)] $B \in C^{\alpha}(\overline{B_r})$, $F_N = 0$ and
    $\sup_{x_N \in (0,r)} (\|\pa_i F_i(\cdot, x_N)\|_{C^{\alpha}(\overline{\pa' B_r})} + \|G(\cdot, x_N)\|_{C^{\alpha}(\overline{\pa' B_r})}) < \infty$
\item[(R32)] $s, t \in C^{\alpha}(\overline{\pa' B_r})$.
\end{enumerate}
Then $x_N^{1-2\ga} \pa_N U \in C^{\min\{\alpha, 2-2\ga\}} (\overline{B_{r/2}})$.
\end{lemma}
\begin{proof}
The first equation in \eqref{eq-reg} can be rewritten as
\[- \pa_N (x_N^{1-2\ga} A_{NN} \pa_N U) = x_N^{1-2\ga} \left[\pa_i (A_{ij} \pa_j U) - BU + G + \pa_i F_i \right] =: x_N^{1-2\ga} Q \quad \text{in } B_r.\]
Hence it holds that
\[- A_{NN}(x_0) \cdot x_{N0}^{1-2\ga} \pa_N U(x_0) = \kappa^{-1} A_{NN}(\bx_0,0) (sU + t)(\bx_0,0)
+ \underbrace{\int_0^{x_{N0}} x_N^{1-2\ga} Q(\bx_0, x_N) dx_N}_{=: \mathcal{Q}(x_0)}\]
for any point $x_0 = (\bx_0, x_{N0}) \in \overline{B_r}$. From this relation and the bound that $\Lambda_1 \le A_{NN}(x_0) \le \Lambda_2$, coming from the assumption on $A$, we immediately observe that
\begin{multline*}
\|x_N^{1-2\ga} \pa_N U\|_{L^{\infty}(B_r)} \le C \( \|U\|_{L^{\infty}(B_r)} + \|t\|_{L^{\infty}(\pa' B_r)} \right. \\
\left. + \sum_{\ell = 1}^2 \|\nabla_{\bx}^{\ell} U\|_{L^{\infty}(B_r)} + \sum_{i=1}^n \|\pa_i F_i\|_{L^{\infty}(B_r)} + \|G\|_{L^{\infty}(B_r)} \)
\end{multline*}
where $C > 0$ depends only on $n$, $\ga$, $r$, $\Lambda_1$, $\Lambda_2$, $\|s\|_{L^{\infty}(\pa' B_r)}$ and $\|B\|_{L^{\infty}(B_r)}$.
In addition, a simple computation reveals that there exists $C > 0$ counting only on $\ga$ and $r$ such that
\[\|\mathcal{Q}\|_{C^{\min\{\alpha, 2-2\ga\}}(\overline{B_{r/2}})} \le C \sup_{x_N \in (0, r/2)} \|Q(\cdot, x_N)\|_{C^{\alpha}(\overline{\pa' B_{r/2}})}.\]
Therefore
\begin{align*}
&\ \|x_N^{1-2\ga} \pa_N U\|_{C^{\min\{\alpha, 2-2\ga\}} (\overline{B_{r/2}})} \\
&\le C\(\|U\|_{C^{\alpha}(\overline{\pa' B_r})} + \|t\|_{C^{\alpha}(\overline{\pa' B_r})}
+ \sup_{x_N \in (0, r/2)} \|Q(\cdot, x_N)\|_{C^{\alpha}(\overline{\pa' B_{r/2}})} \) \\
&\le C\( \sum_{\ell = 0}^2 \|\nabla_{\bx}^{\ell} U\|_{C^{\alpha}(\overline{B_r})} + \|t\|_{C^{\alpha}(\overline{\pa' B_r})} \)\\
&\ + C \sup_{x_N \in (0,r)} \( \sum_{i=1}^n \|\pa_i F_i (\cdot, x_N) \|_{C^{\alpha}(\overline{\pa' B_r})} + \|G (\cdot, x_N) \|_{C^{\alpha}(\overline{\pa' B_r})}\)
\end{align*}
where the constant $C > 0$ depends only on $n$, $\ga$, $r$, $A$, $\|s\|_{C^{\alpha}(\overline{\pa' B_r})}$ and $\|B\|_{C^{\alpha}(\overline{B_r})}$.
\end{proof}

\subsection{Proof of Theorem \ref{thm-reg}}
Finally, collecting all the results obtained in the previous subsection together, we deduce a regularity result on our main equation \eqref{eq-ps} or its extension \eqref{eq-ps-ext}.
It particularly validates Theorem \ref{thm-reg}.
\begin{prop}
Assume that $(X^N, g^+)$ is an asymptotically hyperbolic manifold, $(M^n, [\hh])$ is its conformal infinity,
$\rho$ is the geodesic boundary defining function of $(M, \hh)$ and $\bg = \rho^2 g^+$.
In addition, we suppose that the mean curvature $H$ on $(M, \hh)$ as a submanifold of $(\ox, \bg)$ vanishes.
If $f \in C^{\infty}(M)$ and $U \in \hx$ weakly solves \eqref{eq-ps-ext},
then the trace $u \in H^{\ga}(M)$ of $U$ on $M$ is in fact of class $C^{\infty}(M)$ and a classical solution to \eqref{eq-ps}.
Moreover, $\nabla_{\bx}^m U$ and $\rho^{1-2\ga} \pa_{\rho} U$ are H\"older continuous on $\ox$ for every $m \in \mn$.
\end{prop}
\begin{proof}
The standard elliptic estimate works if $\ga = \frac{1}{2}$ as confirmed by Cherrier \cite{Ch}, so we assume that $\ga \in (0,1) \setminus \{\frac{1}{2}\}$.

\medskip
Fix any $y \in M$ and choose a smooth metric $\hh_y \in [\hh]$ on $M$ such that $|\hh_y| = 1$ around $y$ whose existence was guaranteed by Cao \cite{Ca} and G\"unther \cite{Gun}.
Let $\rho_y$ be the associated geodesic defining function in $X$, $\bg_y = \rho_y^2 \hh_y$ be a smooth metric on $\ox$
and $w_y$ be a positive smooth function on $M$ such that $\hh_y = w_y^{4 \over n-2\ga} \hh$ on $M$.
Since $|\bg_y| = |\hh_y| = 1$ on $M$ near $y$, it holds that $|\bg_y| = 1 + O(\rho_y)$ in $X$ near $y$.
As a matter of fact, we have that $|\bg_y| = 1 + O(\rho_y^2)$ by the condition $H = 0$ on $M$.

In view of Lemma \ref{lemma-GQ-ext}, we know that $P_{\hh}^{\ga} u = f u^{2^*}$ in $H^{-\ga}(M)$.
By the conformal covariance property \eqref{eq-sym}, the function $u_y := w_y^{-1} u \in H^{\ga}(M)$ weakly solves $P_{\hh_y}^{\ga} (u_y) = f u_y^{2^*}$ on $M$.
Besides, there is a solution $U_y \in \hx$ to \eqref{eq-ps-ext}, in which the subscript $y$ is attached suitably, such that $U_y = u_y$ on $M$.

Denote by $x$ the Fermi coordinate on $(\ox, \bg_y)$ around $y$.
We assume that it is defined in the geodesic half-ball $B_{\bg}(y, r) := \{\tilde{y} \in X: \text{dist}_{\bg}(y, \tilde{y}) < r\}$
identified with $B_r = B^N_+(0,r) \subset \mr^N_+$.
Then, owing to Remark \ref{rmk-E} and the assumption $H = 0$ on $M$, the term $E(\rho_y) = E(x_N)$ can be expressed as $x_N^{1-2\ga} B$ for some function $B \in C^{\infty}(\overline{B_r})$.
Consequently, if we set $A_{ab} = \sqrt{|\bg_y|}\, \bg_y^{ab}$ and $F_a = G = 0$ on $B_r$,
then the equation \eqref{eq-ps-ext} of $U_y$ can be described as \eqref{eq-reg} with $s = u_y^{2^*-1}$, $t = 0$ on $\pa' B_r$ or \eqref{eq-reg-2} with $\beta = 2^*$,
where conditions (R1), (R12), (R2), (R23), (R24), (R3) and (R32) are all fulfilled.

We first regard \eqref{eq-ps-ext} as \eqref{eq-reg}.
By employing the Sobolev trace inequality, we see that $s \in L^{n \over 2\ga}(\pa' B_r)$.
Thus we can apply Lemma \ref{lemma-reg-1} (1) to derive that $s \in L^{p_1}(\pa' B_{r/2})$ for some $p_1 > \frac{n}{2\ga}$.
Because of Lemma \ref{lemma-reg-1} (2), this implies that $U_y \in C^{\alpha}(\overline{B_{r/4}})$ for some $\alpha \in (0,1)$.
Furthermore, regarding \eqref{eq-ps-ext} as \eqref{eq-reg-2}, we realize from Corollary \ref{cor-reg-2} that $\nabla_x^m u_y(0)$ exists for all $m \in \mn$.
This means that $u$ is infinitely differentiable at $y$, and for $y$ is arbitrary, it leads that $u \in C^{\infty}(M)$.

The last assertion in the statement can be proved via \eqref{eq-v_m-form} and Lemma \ref{lemma-reg-3}.
This concludes the proof.
\end{proof}

\appendix
\section{End-point Case of the Fractional Yamabe problem}\label{sec-Y}
This section is devoted to proving Theorem \ref{thm-Y-end}.
By Proposition \ref{prop-cpt}, it is enough to verify \eqref{eq-str}.
Three mutually exclusive cases (B$^{\prime}$), (C$^{\prime}$) and (D$^{\prime}$) will be handled in Propositions \ref{prop-Y-end-1}, \ref{prop-Y-end-2} and \ref{prop-Y-end-3}, respectively.
Our proof will be rather sketchy, so we ask the reader to consult the papers \cite{Al, Es2, Es, Ma2, Ma} on the boundary Yamabe problem for more detailed explanation under analogous settings.

\begin{prop}\label{prop-Y-end-1}
Suppose that $n = 3$, $\ga = \frac{1}{2}$, $(M, \hh)$ is the non-umbilic boundary of $(\ox, \bg)$ and \eqref{eq-R-dec} is valid. Then \eqref{eq-str} holds.
\end{prop}
\begin{proof}
We fix a non-umbilic point $y \in M$ and set $\Psi_{1\ep}$ as in \eqref{eq-psi_1}. It suffices to prove that
\begin{equation}\label{eq-Y-end-1}
I^{\ga}(\chi_{r_2} (W_{\ep,0} + \Psi_{1\ep})) \le \ola^{\ga}(\mh^4, [\hh_c]) + O(\ep^2)
\end{equation}
for some $C_1 \in \mr$ where $\chi_{r_2}$ is the cut-off function defined after \eqref{eq-chi}.
Then the argument of Marques \cite[pp. 400--403]{Ma} will give
\begin{equation}\label{eq-Ma}
I^{\ga}(\chi_{r_2} (W_{\ep,0} + \Psi_{1\ep} + C_1' \text{II}_{ij}(y) x_ix_j \pa_r W_{\ep,0})) \le \ola^{\ga}(\mh^4, [\hh_c]) - C \ep^2 |\log{\ep}| + O(\ep^2)
\end{equation}
for some $C > 0$ and small $C_1' \in \mr$, which in particular tells us that \eqref{eq-str} holds.

We may assume that the metrics $\bg$ on $\ox$ and $\hh$ on $M$ satisfy \eqref{eq-hh_0}. As in \cite[Lemma 2.6]{KMW},
using Lemma \ref{lemma-bg-exp}, \eqref{eq-bubble-12} and the identity $\pa_i W_{\ep, 0}(x) = x_i r^{-1} \pa_r W_{\ep, 0}(x)$ which holds for all $x \in \mr^N_+$, we calculate
\begin{align*}
\int_X |\nabla (\chi_{r_2} W_{\ep, 0})|_{\bg}^2 dv_{\bg} &= \int_{\mr^4_+} |\nabla W_{\ep, 0}|^2 dx \\
&\ + (3\text{II}_{ik}(y) \text{II}_{kj}(y) + R_{iNjN}[\bg](y)) \int_{B^4_+(0,r_2)} x_N^2 \pa_i W_{\ep, 0} \pa_j W_{\ep, 0} dx \\
&\ - {1 \over 2} \(\|\text{II}(y)\|^2 + R_{NN}[\bg](y)\) \int_{B^4_+(0,r_2)} x_N^2 |\nabla W_{\ep, 0}|^2 dx + O(\ep^2) \\
&= \int_{\mr^4_+} |\nabla W_{1,0}|^2 dx + {7 \over 12} \|\text{II}(y)\|^2 \ep^2 \int_{B^4_+(0,r_2\ep^{-1})} x_N^2 |\nabla_{\bx} W_{1,0}|^2 dx \\
&\ + {1 \over 8} \|\text{II}(y)\|^2 \ep^2 \int_{B^4_+(0,r_2\ep^{-1})} x_N^2 |\nabla W_{1,0}|^2 dx + O(\ep^2) \\
&= \int_{\mr^4_+} |\nabla W_{1,0}|^2 dx + {5\pi \over 48} \alpha_{3, \frac{1}{2}}^2 |\mathbb{S}^2| \|\text{II}(y)\|^2 \ep^2 \log\({r_2 \over \ep}\) + O(\ep^2)
\end{align*}
where $\nabla_{\bx} W_{1,0} = (\pa_1 W_{1,0}, \cdots, \pa_n W_{1,0})$ and $|\mathbb{S}^2| = 4\pi$ is the surface measure of the unit 2-sphere.
Owing to \eqref{eq-E} and Lemma \ref{lemma-bg-exp}, we also have
\[E(x_N) = \(\|\text{II}(y)\|^2 + R_{NN}[\bg](y)\) (1+ O(|x|)) = -{1 \over 4} \|\text{II}(y)\|^2 (1+ O(|x|))\]
in $B^4_+(0, 2r_2)$. Thus
\begin{align*}
\int_X E(\rho) (\chi_{r_2} W_{\ep, 0})^2 dv_{\bg} &= -{1 \over 4} \|\text{II}(y)\|^2 \ep^2 \int_{B^4_+(0,r_2)} W_{1,0}^2 dx + O(\ep^2) \\
&= - {\pi \over 16} \alpha_{3, \frac{1}{2}}^2 |\mathbb{S}^2| \|\text{II}(y)\|^2 \ep^2 \log\({r_2 \over \ep}\) + O(\ep^2).
\end{align*}
Consequently, we derive from the definition \eqref{eq-I^ga} of $I^{\ga}$ that
\[I^{\ga}(\chi_{r_2} W_{\ep,0}) = \int_{\mr^4_+} |\nabla W_{1,0}|^2 dx + {\pi \over 24} \alpha_{3, \frac{1}{2}}^2 |\mathbb{S}^2| \|\text{II}(y)\|^2 \ep^2 \log\({r_2 \over \ep}\) + O(\ep^2).\]
Moreover, a direct computation shows
\begin{align*}
&\ I^{\ga}(\chi_{r_2} (W_{\ep,0} + \Psi_{1\ep})) \\
&= I^{\ga}(\chi_{r_2} W_{\ep,0}) + 4 \text{II}_{ij}(y) \int_{B^4_+(0,r_2)} \pa_i W_{\ep} \pa_j \Psi_{1\ep} dx + \int_{B^4_+(0,r_2)} |\nabla \Psi_{1\ep}|^2 dx + O(\ep^2) \\
&= \int_{\mr^4_+} |\nabla W_{1,0}|^2 dx + \(1 + 4C_1 + 4C_1^2\) {\pi \over 24} \alpha_{3, \frac{1}{2}}^2 |\mathbb{S}^2| \|\text{II}(y)\|^2 \ep^2 \log\({r_2 \over \ep}\) + O(\ep^2);
\end{align*}
compare with \cite[Proposition 2.8]{KMW}. Hence, choosing $C_1 = -\frac{1}{2}$, we observe the validity of \eqref{eq-Y-end-1}. This completes the proof.
\end{proof}
\noindent Taking the above proposition into consideration, one can guess that the local geometry on $M$ still allows us to derive \eqref{eq-str}
when $n = 3$, $\ga \in (0,\frac{1}{2})$, $(M, \hh)$ is the non-umbilic boundary of $(\ox, \bg)$ and \eqref{eq-R-dec} holds.
However, there seems a computational difficulty in employing test functions whose forms are similar to that of the function in the previous proof.

\begin{prop}\label{prop-Y-end-2}
Suppose that $n = 4$, $\ga = \frac{1}{2}$, $(M, \hh)$ is the umbilic boundary of $(\ox, \bg)$, $R_{\rho\rho;\rho}[\bg](y) < 0$ for some $y \in M$ and \eqref{eq-R-dec} is valid. Then \eqref{eq-str} holds.
\end{prop}
\begin{proof}
Suppose that the metrics $\bg$ and $\hh$ satisfy \eqref{eq-hh_0}.
Since $\text{II} = 0$ on $M$ by the assumption, we know that $R_{NN}[\bg](y) = 0$.
By \eqref{eq-E} and Lemma \ref{lemma_metric_2},
\[E(x_N) = {3 \over 2} \(R_{NN;i}[\bg] x_i + {1 \over 2} R_{NN;N}[\bg] x_N + O(|x|^2)\)\]
in $B^5_+(0,2r_2)$. It follows that
\begin{align*}
I^{\ga}(\chi_{r_2} W_{\ep,0}) &= \int_{\mr^5_+} |\nabla W_{1,0}|^2 dx + {3 \over 4} R_{NN;N}[\bg](y) \ep^3 \int_{B^5_+(0,r_2 \ep^{-1})} x_N W_{1,0}^2 dx \\
&\ + R_{NN;N}[\bg](y) \ep^3 \int_{B^5_+(0,r_2 \ep^{-1})} x_N^3 \( {1 \over 12} |\nabla_{\bx} W_{1,0}|^2 - {1 \over 6} |\nabla W_{1,0}|^3\) dx + O(\ep^3) \\
&= \int_{\mr^5_+} |\nabla W_{1,0}|^2 dx + {3 \over 32} \alpha_{4, \frac{1}{2}}^2 |\mathbb{S}^3| R_{NN;N}[\bg](y) \ep^3 \log\({r_2 \over \ep}\) + O(\ep^3).
\end{align*}
From this, one can conclude that \eqref{eq-str} holds.
\end{proof}

\begin{prop}\label{prop-Y-end-3}
Assume that $n = 5$, $\ga = \frac{1}{2}$, $(M, \hh)$ is the umbilic non-locally conformally flat boundary of $(\ox, \bg)$ and \eqref{eq-R-dec-2} is valid. Then \eqref{eq-str} holds.
\end{prop}
\begin{proof}
We assume that $\hh$ is the representative of its conformal class satisfying all the properties listed in Lemma \ref{lemma-g-mtr-2} and the Weyl tensor at $y \in M$ is nontrivial.
Then a tedious but straightforward computation gives
\begin{align*}
&\int_X |\nabla (\chi_{r_2} W_{\ep, 0})|_{\bg}^2 dv_{\bg} \\
&= \int_{\mr^6_+} |\nabla W_{\ep, 0}|^2 dx + {1 \over 2} R_{iNjN;kl}[\bg] \int_{B^6_+(0,r_2)} x_N^2 x_kx_l \pa_i W_{\ep, 0} \pa_j W_{\ep, 0} dx \\
&\ + {1 \over 12} (R_{iNjN;NN}[\bg] + 8 R_{iNkN}[\bg] R_{kNjN}[\bg]) \int_{B^6_+(0,r_2)} x_N^4 \pa_i W_{\ep, 0} \pa_j W_{\ep, 0} dx \\
&\ - {1 \over 4} R_{NN;ij}[\bg] \int_{B^6_+(0,r_2)} x_N^2x_ix_j |\nabla W_{\ep, 0}|^2 dx \\
&\ - {1 \over 24} \(R_{NN;NN}[\bg] + 2(R_{iNjN}[\bg])^2\) \int_{B^6_+(0,r_2)} x_N^4 |\nabla W_{\ep, 0}|^2 dx + O(\ep^4) \\
&= \int_{\mr^6_+} |\nabla W_{1,0}|^2 dx + {\pi \over 640} \alpha_{5, \frac{1}{2}}^2 |\mathbb{S}^4|
\left[-{1 \over 8} \|W\|^2 - {1 \over 2} R_{;NN}[\bg] + 2(R_{ij}[\bg])^2 \right] \ep^4 \log\({r_2 \over \ep}\) \\
&\ + {\pi \over 640} \alpha_{5, \frac{1}{2}}^2 |\mathbb{S}^4| \left[{5 \over 24} \|W\|^2 - {3 \over 2} R_{;NN}[\bg] \right] \ep^4 \log\({r_2 \over \ep}\) + O(\ep^4)
\end{align*}
where all tensors are evaluated at $y \in M$. Furthermore, we obtain  from \eqref{eq-E} and Lemma \ref{lemma_metric_2} that
\begin{multline*}
E(x_N) = 2 R_{NN;i}[\bg] x_i + R_{NN;ij}[\bg] x_ix_j + R_{NN;Ni}[\bg] x_Nx_i \\
+ {1 \over 3} \(R_{NN;NN}[\bg] + 2(R_{iNjN}[\bg])^2\) x_N^2 + O(|x|^3)
\end{multline*}
in $B^6_+(0,2r_2)$ and so
\[\int_X E(\rho) (\chi_{r_2} W_{\ep, 0})^2 dv_{\bg} = {\pi \over 640} \alpha_{5, \frac{1}{2}}^2 |\mathbb{S}^4| \left[-{5 \over 12} \|W\|^2 + 2 R_{;NN}[\bg] \right] \ep^4 \log\({r_2 \over \ep}\) + O(\ep^4)\]
in light of Lemma \ref{lemma-g-mtr-2} (3) and (4). As a result,
\[I^{\ga}(\chi_{r_2} W_{\ep,0}) = \int_{\mr^6_+} |\nabla W_{1,0}|^2 dx + {\pi \over 640} \alpha_{5, \frac{1}{2}}^2 |\mathbb{S}^4|
\left[-{1 \over 3} \|W\|^2 + 2(R_{ij}[\bg])^2 \right] \ep^4 \log\({r_2 \over \ep}\) + O(\ep^4).\]

We now recall the function $\Psi_{2\ep}$ defined in \eqref{eq-psi_2}. With this, one can compute
\begin{equation}\label{eq-Al}
\begin{aligned}
&\ I^{\ga}(\chi_{r_2} (W_{\ep,0} + \Psi_{2\ep})) \\
&= I^{\ga}(\chi_{r_2} W_{\ep,0}) + 2 R_{iNjN}[\bg] \int_{B^6_+(0,r_2)} x_N^2 \pa_i W_{\ep} \pa_j \Psi_{2\ep} dx + \int_{B^6_+(0,r_2)} |\nabla \Psi_{2\ep}|^2 dx + O(\ep^4) \\
&= \int_{\mr^6_+} |\nabla W_{1,0}|^2 dx + {\pi \over 640} \alpha_{5, \frac{1}{2}}^2 |\mathbb{S}^4|
\left[-{1 \over 3} \|W\|^2 + (2 + 24C_2 + 56C_2^2) (R_{ij}[\bg])^2 \right] \ep^4 \log\({r_2 \over \ep}\) \\
&\ + O(\ep^4).
\end{aligned}
\end{equation}
If we take $C_2 = -\frac{3}{14}$, then $2 + 24C_2 + 56C_2^2 = -\frac{4}{7} < 0$, yielding the validity of \eqref{eq-str}.
\end{proof}

\end{document}